\documentclass[11pt,twoside]{amsart}
\usepackage{amssymb,amsmath,amsthm}
\usepackage{graphicx}
\usepackage[colorlinks,linkcolor=blue,citecolor=blue]{hyperref}

\numberwithin{equation}{section}
\newtheorem{theorem}{Theorem}[section]
\newtheorem{lemma}{Lemma}[section]
\newtheorem{corollary}{Corollary}[section]
\newtheorem{proposition}{Proposition}[section]
\newtheorem{remark}{Remark}[section]
\newtheorem{definition}{Definition}[section]
\newtheorem{assumption}{Assumption}[section]
\theoremstyle{definition}
\newtheorem{example}{Example}[section]
\newtheorem{notation}{Notation}[section]

\newcommand{\crit}{\operatorname{Crit}}
\newcommand{\res}{\operatorname{r\text{\'e}sit}}
\renewcommand{\Re}{\operatorname{Re}}

 \usepackage{fancyhdr} 
\title{Boundary of the central hyperbolic component I: dynamical properties}

  \author{Jie Cao}
\address{(J. Cao) School of Mathematical Sciences, Shenzhen University, Shenzhen, 518060, China}
\email{mathcj@foxmail.com}

  \author{Xiaoguang Wang}
\address{(X. Wang) School of Mathematical Sciences, Zhejiang University, Hangzhou, 310027, China}
\email{wxg688@163.com}

 \author{Yongcheng Yin}
\address{(Y. Yin) School of Mathematical Sciences, Zhejiang University, Hangzhou, 310027, China}
\email{yin@zju.edu.cn}

\begin{document}

\begin{abstract}
We study the dynamics of polynomial maps on the boundary of the central hyperbolic component $\mathcal H_d$. We prove the local connectivity of Julia sets and a rigidity theorem for maps on the regular part of $\partial\mathcal H_d$.  Our proof is based on the construction of Fatou trees and employs the puzzle technique as a key methodological framework.
These results are applicable to a larger class of maps for which the maximal Fatou trees equal the filled Julia sets.


\end{abstract}

 \subjclass[2020]{Primary 37F45; Secondary 37F10, 37F15}


\keywords{local connectivity, rigidity, Fatou tree, limb decomposition, puzzle}



\date{\today}

\maketitle





\section{Introduction}
Fix an integer $d\geq3$. 
The space of degree $d$ polynomials, up to affine conjugacy, can be written in the monic and $0$-fixed form:   $$\mathcal P_d=\left\{f(z)=a_1 z+a_2z^2+\cdots+a_{d-1}z^{d-1}+z^d \mid (a_1, a_2 \dots, a_{d-1})\in \mathbb C^{d-1}\right\}.$$ 
A polynomial $f\in\mathcal{P}_d$ is \emph{hyperbolic} if the orbit of each critical point tends to infinity or a bounded attracting cycle. 
All the hyperbolic polynomials form an open
subset of $\mathcal{P}_d$, and each connected component is called a \emph{hyperbolic component}. 
The \emph{central hyperbolic component} $\mathcal H_d$ in $\mathcal P_d$ is the hyperbolic component containing the map $z\mapsto z^d$. 
For $f\in\mathcal H_d$, all the critical points of $f$ in $\mathbb{C}$ are contained in the Fatou component containing $0$. 
While the polynomials in $\mathcal H_d$ have the simplest dynamical behavior, those on the boundary of $\mathcal H_d$ can exhibit fruitful dynamics.

It is well-known that $\mathcal H_d$ is a topological cell \cite{MilHyperCompo}. 
The family $\mathcal{P}_d$ consists of polynomials with at least two free critical points, which lead to the complexity of $\partial \mathcal{H}_d$. 
For $\partial\mathcal H_3$, Petersen and Tan Lei \cite{PT09} study its tame part, and Blokh et al. \cite{BOPT} study its combinatorial model.
Luo \cite{Luo} gives a combinatorial classification of the geometrically finite maps on $\partial\mathcal H_d$, and shows that $\overline{\mathcal H_d}$ is not a topological manifold with boundary if $d\geq4$. 
By a recent result of Gao, X. Wang and Y. Wang \cite{GWW}, the boundary $\partial \mathcal{H}_d$ has full Hausdorff dimension $2d-2$.  

It is worth noting that certain complex one-dimensional slices of $\mathcal H_d$ have favorable properties.
For example, the intersection of $\mathcal H_d$ and $\{\lambda z +z^d\mid \lambda\in \mathbb C\}$ is the unit disk. 
As a comparison, the intersection of $\mathcal H_d$ and $\{a z^{d-1}+z^d\mid a\in \mathbb C\}$ is a Jordan domain with fractal boundary; see \cite{Faught} for $d=3$ and \cite{Roesch} for $d\geq3$.

The boundary of $\mathcal H_d$ is divided into two parts: the regular part and the singular part (or the tame part and the wild part). That is, 
\begin{align*}
\partial_{\rm reg}\mathcal H_d &=\{f\in\partial\mathcal H_d\mid |f'(0)|<1\},\\
\partial_{\rm sing}\mathcal H_d &=\{f\in\partial\mathcal H_d\mid |f'(0)|=1\}.
\end{align*}
The regular part $\partial_{\rm reg}\mathcal H_d$ is large (with Hausdorff dimension $2d-2$), while the singular part $\partial_{\rm sing}\mathcal H_d$ is small (with Hausdorff dimension at most $2d-3$). 

Our first main theorem shows the tameness of the Julia sets for the maps on $\partial_{\rm reg}\mathcal H_d$. 

\begin{theorem}
\label{J-lc}
For any $f\in \partial_{\rm reg}\mathcal H_d$, the Julia set $J(f)$ is locally connected. 
\end{theorem}

\begin{remark}
Theorem \ref{J-lc} is false for $f\in\partial_{\rm sing}\mathcal H_d$, which has a Cremer point at $0$ (see \cite[Corollary 18.6]{Mil}).
\end{remark}

Our second main theorem establishes the rigidity theorem. 

\begin{theorem}
\label{hybrid4H}
Let $f, g\in \partial_{\rm reg}\mathcal H_d$. If there is a homeomorphism $\phi: \mathbb C\to \mathbb C$ such that $\phi\circ f=g\circ \phi$ and $\phi$ is holomorphic in the Fatou set with normalization $\phi'(\infty)=1$, then $f=g$. 
\end{theorem}

Rigidity, which means that a weak conjugacy implies a stronger conjugacy, is a fundamental topic in dynamical systems (see \cite{McM, Lyu}). 
Progresses have been made for 
real rational maps \cite{Shen}, 
Fibonacci maps \cite{Smania}, 
real polynomials \cite{KSS}, 
rational maps with Cantor Julia sets \cite{Zhai,YZ}, 
non-renormalizable polynomials \cite{KS}, 
proper box mappings \cite{Peng-Tan}, 
non-renormalizable Newton maps \cite{RYZ,DS}, 
bounded type cubic Siegel polynomials \cite{YZh}, etc. 
Beyond independent interests, 
rigidity also serves as a powerful tool to study the hyperbolic component boundaries \cite{QRWY,RWY,Wang21}. 
In the sequel to this paper, rigidity is applied to resolve a boundary extension problem and to study the distribution of cusps.

To prove the local connectivity and the rigidity, we will use the puzzle technique and Fatou trees. The puzzle technique, initiated by Yoccoz and refined over decades, has become a mature tool in complex dynamics. In the following, we outline the construction of Fatou trees. 

For $f\in\mathcal P_d$, let $U_f(0)$ denote the Fatou component of $f$ containing $0$ if $|f'(0)|<1$. 
When $f_n\in \mathcal H_d$ tends to some $f\in \partial_{\rm reg}\mathcal H_d$, the fixed point $0$ keeps attracting, and the pointed disk $(U_{f_n}(0),0)$ converges to $(U_f(0),0)$ in the sense of Carath\'eodory topology. \footnote{See \cite[\S5.1]{McM} for basic facts of the Carath\'eodory topology of pointed disks.}
With many pieces of $U_{f_n}(0)$ pinching off, at least one and at most $d-2$ critical points leave $U_f(0)$. To control these critical points, the Fatou components grow from $U_f(0)$ outward successively. Finally, taking closure gives the Fatou tree $Y^1(f)$ of level 1, which controls more critical points than $U_f(0)$. Some critical points may be so far from $U_f(0)$ that $Y^1(f)$ cannot control them either. However, the Fatou tree $Y^k(f)$ of higher level can control more critical points, and the Fatou tree $Y^{k(f)}(f)$ of highest level $k(f)$ will control all the critical points of $f$. 
See \S\ref{level-1-construction} for the precise construction of the Fatou tree of level $1$ and 
\S\ref{subsec-tree-k} for general levels; see also Example \ref{example-crit}. 

\begin{theorem} 
\label{Ymax=Kf}
For any $f\in \partial_{\rm reg}\mathcal H_d$, the maximal Fatou tree $Y^{k(f)}(f)$ is equal to the filled Julia set $K(f)$. 

Further, for each limit point $y$ in $Y^{k(f)}(f)$ (i.e. $y\in Y^{k(f)}_\infty(f)$), there is exactly one external ray landing at $y$. 
\end{theorem}


For a polynomial $f$ with an attracting fixed point at $0$, we can also construct the Fatou trees of $f$ in the same way. 
By Example \ref{example-resit}, there is $f\in\mathcal{P}_d\setminus \overline{\mathcal{H}_d}$ such that $Y^{k(f)}(f)=K(f)$.  
That is, the converse of the first statement of Theorem \ref{Ymax=Kf} may not hold. 

The attracting Fatou component $U_f(0)$, whose closure can be viewed as the Fatou tree of level $0$, is proved to be a Jordan domain by Roesch and Yin \cite{RY}. 
Furthermore, the filled Julia set $K(f)$ has a limb decomposition around $U_f(0)$. Each limb is a continuum attaching at $\partial U_f(0)$. 
Parallel to this, we will use the puzzle technique to show the local connectivity of $Y^1(f)$. 
Inductively, we can define the Fatou tree $Y^k(f)$ of higher level and prove its local connectivity.

\begin{theorem} 
\label{thm-tree-level-k} 
Let $f\in\mathcal{P}_d$ satisfy that $0$ is attracting and the Julia set $J(f)$ is connected. Then for each $0\leq k\leq k(f)$, the followings hold. 
\begin{enumerate}
\item 
\label{thm-tree-level-k-lc}
The Fatou tree $Y^k(f)$ is locally connected. 

\item 
\label{thm-tree-level-k-limb}
\begin{enumerate}
\item There is a unique decomposition $$K(f)=Y^k(f)^\circ\sqcup \bigsqcup_{y\in \partial Y^k(f)} L_y$$ such that $L_y\cap Y^k(f)=\{y\}$ and $L_y$ is a continuum for any $y\in \partial Y^k(f)$. Here $Y^k(f)^\circ$ is the interior of $Y^k(f)$. 

\item 
For any $y\in \partial Y^k(f)$, there exist external rays landing at $y$. If $L_y\neq \{y\}$, there are finitely many sectors with root $y$ separating $L_y$ from $Y^k(f)$. 

\item 
\label{thm-tree-level-k-diam}
The diameters of $L_y$'s tend to zero. 
\end{enumerate}

\item 
\label{thm-tree-level-k-ray}
Let $y\in \partial Y^k(f)$. If $y$ is preperiodic, then it is either pre-repelling or pre-parabolic; 
if $y$ is wandering, then there are exactly $\prod_{n\in\mathbb{N}}\deg(f, f^n(y))$ external rays landing at $y$. 
\end{enumerate}
\end{theorem}

\begin{remark}
Theorems \ref{J-lc}, \ref{hybrid4H} and \ref{Ymax=Kf} also hold for $f\in\partial_{\rm sing}\mathcal H_d$, which has a parabolic fixed point at $0$. 

Theorem \ref{thm-tree-level-k} also holds for $f\in\mathcal P_d$ with a parabolic fixed point at $0$ and connected Julia set. 
\end{remark}

The paper is organized as follows. 

In \S\ref{sec-limb-Jordan}, we review the results about bounded Fatou components in \cite{RY}. We consider the limb decomposition of a full continuum around a Jordan disk, and show that the existence of cut rays is equivalent to the convergence of limb diameter (Lemma \ref{iff-diam-0}). 
This is useful for the proof of Theorem \ref{thm-tree-level-k}(\ref{thm-tree-level-k-diam}). 

In \S\ref{sec-tree-level-1}, we construct the Fatou tree of level 1, and give a limb decomposition of the filled Julia set around it. To prove its local connectivity, it suffices to check that each limb intersects it at a single point (Lemma \ref{Y1-lc}). The preperiodic case follows from \cite[Lemma 3.9]{Kiwi01} directly.  

In \S\ref{sect-wandering}, we address the wandering case through the puzzle technique.  
To make the discussion clearer, we focus on the nests of puzzle pieces. We also employ the elevator condition from \cite{RYZ}, which is stronger than the bounded degree condition. For those nests intersecting the Fatou tree, these two conditions are equivalent (Lemma \ref{BD2elevator}). 

In \S\ref{sec-lc}, we complete the proof of the local connectivity of the Fatou tree of level 1, and give more properties of the limb decomposition around it. 

In \S\ref{subsec-tree-k}, we construct the Fatou tree of higher level (including the maximal Fatou tree), and prove Theorems \ref{thm-tree-level-k} by induction. 
In \S\ref{subsec-tree-kf}, we prove Theorem \ref{Ymax=Kf} by constructing polynomial-like renormalization for the maximal Fatou tree (Lemma \ref{ren-maximal}). 
Then Theorem \ref{J-lc} follows from Theorems \ref{Ymax=Kf} and \ref{thm-tree-level-k} directly. 
In \S\ref{subsec-example}, we give non-hyperbolic polynomials $f$ with $Y^{k(f)}(f)=K(f)$ but $f\notin \partial\mathcal H_d$. 
In \S\ref{subsec-parab}, we extend Fatou trees to $f\in\mathcal P_d$ with a parabolic fixed point at $0$. 

In \S\ref{sec-rigidity}, we prove the rigidity for polynomials with $Y^{k(f)}(f)=K(f)$. As a corollary of it and Theorem \ref{Ymax=Kf}, we have Theorem \ref{hybrid4H}. 
The proof is divided into two parts: showing the conjugacy is quasiconformal through the quasiconformality criterion, and proving the Julia set carries no invariant line field by \cite[Proposition 3.2]{Shen}. 

In appendix, we give a proof of the quasiconformality criterion, which is an analogue of \cite[Lemma 12.1]{KSS}. 

\vspace{5pt}

\textbf{Acknowledgments.}
The research is supported by National Key R\&D Program of China (Grants No. 2021YFA1003200 and No. 2021YFA1003202) and National Natural Science Foundation of China (Grants No. 12131016 and No. 12331004). 

The authors thank Guizhen Cui for inspiring the current statement of Lemma \ref{iff-diam-0}. 
They also extend their gratitude to Mika Koskenoja for providing a copy of the archival paper \cite{Kelingos}.

\section{The limb decomposition around a Jordan disk}
\label{sec-limb-Jordan}

In this section, we will consider the limb decomposition of a full continuum around a Jordan disk, which generalizes the limb decomposition of the filled Julia set around a Fatou component. 
Moreover, we will give some necessary condition for $f\in\partial_{\rm reg}\mathcal H_d$. 

\subsection{Topological properties}
\begin{definition}
Let $K\subset\mathbb{C}$ be a full connected compact set with a Jordan domain $U\subset K$. 
A decomposition $K=U\sqcup \bigsqcup_{x\in \partial U} L_{x}$ is called a limb decomposition for $U$ if $L_{x}\cap\overline{U}=\{x\}$ and $L_{x}$ is a connected compact set for any $x\in \partial U$. 

We call $L_{x}$ the limb for $(K,U)$ with root $x$. 
If $K$ has a unique limb decomposition for $U$, we denote the limb for $(K,U)$ with root $x$ by $L_x(K,U)$. 
\end{definition}

\begin{theorem}
[\cite{RY}]
\label{RY}
Let $f$ be a polynomial such that $\deg(f)\geq2$ and $J(f)$ is connected, and let $U$ be a pre-attracting or pre-parabolic bounded Fatou component of $f$. 
Then the following properties hold. 
\begin{enumerate}
\item
\label{U-Jordan}
$U$ is a Jordan domain. 

\item $K(f)$ has a unique limb decomposition $K(f)=U\sqcup \bigsqcup_{x\in \partial U} L_{x}$ for $U$. 

\item 
\label{RY-separate}
When $L_{x}=\{x\}$, there is only one external ray landing at $x$; when $L_{x}\neq\{x\}$, two external rays landing at $x$ separate $L_{x}$ from $U$. 

\item The diameters of limbs tend to zero. 
\end{enumerate}
\end{theorem}

\begin{proof}
Property 1 is \cite[Theorem A]{RY}. 
Note that the repelling periodic points are dense on $\partial U$. Then Lemma \ref{limb-K-U} gives Property 2. 
Property 3 is \cite[Theorem B]{RY}. 
Property 4 follows from Property 3 and Lemma \ref{iff-diam-0}. 
\end{proof}

For a full connected compact set $K\subset\mathbb{C}$, which is not a singleton, there is a unique conformal map $\phi_K:\mathbb{C}\setminus K\to \mathbb{C}\setminus\overline{\mathbb{D}}$ such that $\phi_K(z)/z\to 1$ as $z\to\infty$. The external ray with angle $\theta\in\mathbb{R}/\mathbb{Z}$ is defined to be $R_K(\theta) = \phi_K^{-1}((1,\infty)e^{2\pi i \theta})$. We say $R_K(\theta)$ lands at $x\in\partial K$ if $\overline{R_K(\theta)}\cap \partial K = \{x\}$. 

\begin{lemma}
\label{limb-K-U}
Let $K\subset\mathbb{C}$ be a full connected compact set with a Jordan domain $U\subset K$. 
If the points with external rays landing are dense in $\partial U$, then $K$ has a unique limb decomposition for $U$. 
\end{lemma}

\begin{proof}
Given $x\in\partial U$. Choose two sequences of points on $\partial U$ with external rays landing and tending to $x$ from two sides. Let $L_{x}$ be the intersection of Jordan domains bounded by external rays and external equipotential curves in $\mathbb{C}\setminus K$ and internal rays and internal equipotential curves in $U$. 
This gives the construction of $L_{x}$. The residual is clear. 
\end{proof}


\begin{lemma}
\label{iff-diam-0}
Let $K\subset\mathbb{C}$ be a full connected compact set with a Jordan domain $U\subset K$. Suppose $K$ has a limb decomposition $K=U\sqcup \bigsqcup_{x\in \partial U} L_{x}$. Then the following two conditions are equivalent. 
\begin{enumerate}
\item  When $L_{x}=\{x\}$, there is only one external ray landing at $x$; when $L_{x}\neq\{x\}$, two external rays landing at $x$ separate $L_{x}$ from $U$. 

\item  The diameters of limbs tend to zero. 
\footnote{Compare \cite[Theorem 7]{Roesch}.}
\end{enumerate}
\end{lemma}

\begin{proof} 
$1\Longrightarrow 2$. 
If it is not true, there exists an $\varepsilon>0$ and a non-repeating sequence $\{x_k\}_{k\in\mathbb{N}}\subset \partial U$ such that $\operatorname{diam}(L_{x_k})\geq\varepsilon$ for any $k\in\mathbb{N}$. After passing to a subsequence, the sequence $\{L_{x_k}\}_{k\in\mathbb{N}}$ has a limit $L$ in Hausdorff topology.
There are two cases: either $L\subset \partial U$ or $L\not\subset \partial U$. 

For the first case, $L$ is an arc on $\partial U$. Choose an arc $\gamma$ in the interior of $L$. There are external rays $R_1$ and $R_2$ landing at the two end points $a_1$ and $a_2$ of $\gamma$ respectively. For $x\in\{x_k\}_{k\in\mathbb{N}}\setminus\{a_1,a_2\}$, the set $L_{x}\setminus\{x\}$ 
is contained in one component of $\mathbb{C}\setminus(R_1\cup\gamma\cup R_2)$. 
Hence the Hausdorff limit of $\{L_{x_k}\}_{k\in\mathbb{N}}$ cannot be the whole set $L$. So this case is impossible.

For the second case, take $y\in L\setminus \partial U$. Note that $y\in K\setminus \overline{U}$. By the limb decomposition, the point $y$ is in some limb $L_{a}$. Then there are external rays $R_1$ and $R_2$ landing at $a$, and $L_{a}\setminus\{a\}$ is contained in one component of $\mathbb C\setminus\overline{R_1\cup R_2}$. It follows that $L_{x}$ is contained in the other component of $\mathbb C\setminus\overline{R_1\cup R_2}$ for each $x\in\{x_k\}_{k\in\mathbb{N}}\setminus\{a\}$. Therefore $y$ cannot be an accumulation point of $\{L_{x_k}\}_{k\in\mathbb{N}}$. So this case cannot happen either. 

$2\Longrightarrow 1$. 
First, consider $L_{x}=\{x\}$. Choose a simple path $\gamma_0\subset \mathbb{C}\setminus \overline{U}$  connecting $x$ and $\infty$. We will modify $\gamma_0$ to get a path $\gamma$ in $\mathbb{C}\setminus K$ converging to $x$. 
If $\gamma_0$ intersects with finitely many limbs, let $\gamma$ be a segment of $\gamma_0$ near $x$ in $\mathbb{C}\setminus K$; otherwise, let $\{L_{x_k}\}_{k\in\mathbb{N}}$ be all limbs intersecting with $\gamma_0$. 
Inductively, we will modify $\gamma_k$ to get $\gamma_{k+1}$ as follows. 

Note that $L_{x_k}$ is full. There is conformal map $\phi_k:\mathbb{C}\setminus L_{x_k}\to \mathbb{C}\setminus \overline{\mathbb{D}}$. 
By \cite[Corollary 6.4]{McM}, if $\alpha$ is a path in $\mathbb{C}\setminus L_{x_k}$ converging to a point in $\partial L_{x_k}$, then $\phi_k(\alpha)$ converges to a point in $S^1$. 
Hence $\phi_k(U)$ is a Jordan domain. 
It follows that $\mathbb{C}\setminus \phi_k(\mathbb{C}\setminus K)$ has a limb decomposition 
$$\mathbb{C}\setminus \phi_k(\mathbb{C}\setminus K) = \phi_k(U)\sqcup \overline{\mathbb{D}} \sqcup \bigsqcup_{y\in\partial U\setminus \{x_k\}} \phi_k(L_{y}).$$
Let $a_k$ (resp. $b_k$) be the first intersection point of $\gamma_k$ and $L_{x_k}$ starting from $\infty$ (resp. $x$). By \cite[Corollary 6.4]{McM} again, the path $\phi_k(\gamma_k(a_k,\infty))$ and $\phi_k(\gamma_k(b_k,x))$ converge to points in $S^1$. 
For $r_k>1$ close to $1$ enough, 
there is a unique subarc $\Gamma_k$ of $r_k S^1$ intersecting with $\phi_k(\gamma_k(a_k,\infty))$ (resp. $\phi_k(\gamma_k(b_k,x))$) only at one end $a_k'$ (resp. $b_k'$) and contained in $\mathbb{C}\setminus \phi_k(\mathbb{C}\setminus K)$. 
Here $\gamma_k(a,b)$ is the open segment of $\gamma_k$ between $a$ and $b$ for $a\neq b$ in $\gamma_k\cup\{\infty,x\}$. 
Now let $$\gamma_{k+1} = \gamma_k(\infty,\phi_k^{-1}(a'_k))\cup \phi_k^{-1}(\Gamma_k)\cup \gamma_k(\phi_k^{-1}(b'_k),x).$$

Because $\operatorname{diam}(L_{x_k})\to 0$ as $k\to \infty$, for any $\varepsilon>0$, the modification of $\gamma_k$ is in $\mathbb{D}(x,\varepsilon)$ for $k$ large enough. Therefore $\gamma_k$ has a limit $\gamma$, which is in $\mathbb{C}\setminus K$ and converges to $x$. 

By the Lindel\"of theorem \cite[Theorem 6.3]{McM}, there is an external ray landing at $x$. To show the uniqueness, assume there are two external rays $R_1$ and $R_2$ landing at $x$. 
Let $W_1$ (resp. $W_2$) denote the connected component of $\mathbb{C}\setminus \overline{R_1\cup R_2}$ containing (resp. disjoint from) $U$. For any $y\in \partial U\setminus \{x\}$, we have $L_{y} \subset W_1$. It follows that $K\cap W_2=\emptyset$. This contradicts the F. and M. Riesz theorem \cite[Theorem 6.2]{McM}, which says the radial limits of a Riemann mapping are nonconstant on any set of positive measure. 

Second, consider $L_{x}\neq \{x\}$. 
First of all, we need to construct a Jordan curve separating $\overline{U}\setminus \{x\}$ and $L_{x}\setminus \{x\}$. 
By composing a proper homeomorphism of $\widehat{\mathbb{C}}$, we may assume $U=\widehat{\mathbb{C}}\setminus \overline{\mathbb{D}}$. 
Let $L$ be the convex hull of $L_{x}$ in the plane. 
It follows from $L_{x}\cap \overline{U}=\{x\}$ that $L\cap \overline{U}=\{x\}$. Now it is easy to construct a Jordan curve $\gamma_0$ separating $\overline{U}\setminus \{x\}$ and $L\setminus \{x\}$, which also separates $\overline{U}\setminus \{x\}$ and $L_{x}\setminus \{x\}$. 
Similar to the first case, we can modify $\gamma_0$ to get a Jordan curve $\gamma$ in $(\widehat{\mathbb{C}}\setminus K) \cup\{x\}$ separating $\overline{U}\setminus \{x\}$ and $L_{x}\setminus \{x\}$. Then the Lindel\"of theorem \cite[Theorem 6.3]{McM} gives the desired external rays. 
\end{proof}

\subsection{Dynamical properties}
\begin{theorem}
[\cite{RY}]
\label{dynam-prop}
Let $f$ be a polynomial such that $\deg(f)\geq2$ and $J(f)$ is connected, and let $U$ be a pre-attracting or pre-parabolic bounded Fatou component of $f$. 
\begin{enumerate}
\item If $f(U)=U$, then $f|_{\partial U}$ is topologically conjugate to $z^{\deg(f|_U)}$ on $S^1$. 

\item 
\label{dynam-prop-prep}
Every preperiodic $x\in \partial U$ is either pre-repelling or pre-parabolic.

\item 
\label{dynam-prop-crit}
For any $x\in \partial U$, the limb $L_{x}(K(f),U)$ is not reduced to $\{x\}$ if and only if there is an integer $n\geq 0$ such that $L_{f^n(x)}(K(f),f^n(U))$ contains a critical point.
\end{enumerate}
\end{theorem}

\begin{proof}
These dynamical properties can be deduced from the topological properties (Theorem \ref{RY}) directly. 
\end{proof}

Let $f\in\mathcal{P}_d$ be a polynomial with connected Julia set. 
Then there is a unique \emph{B\"ottcher map} $B_f:\mathbb{C}\setminus K(f)\rightarrow\mathbb{C}\setminus \overline{\mathbb{D}}$ tangent to the identity at $\infty$ and satisfying $B_f(f(z))=B_f(z)^d$. For each angle $\theta\in\mathbb{R/Z}$, the \emph{external ray} $R_f(\theta)$ is defined by $R_f(\theta) = B_f^{-1}((1,\infty)e^{2\pi i\theta})$. Each locus $B_f^{-1}(r S^1)$ with $r>1$ is called an \emph{equipotential curve}. 

For $\theta_1\neq \theta_2$ in $\mathbb{R/Z}$, if the external rays $R_f(\theta_1)$ and $R_f(\theta_2)$ land at a common point $x$, the \emph{sector} $S_f(\theta_1, \theta_2)$ is defined to be the connected component of $\mathbb C\setminus (R_f(\theta_1)\cup R_f(\theta_2)\cup \{x\})$, containing the external rays $R_f(\theta)$ with angles $\theta\in(\theta_1,\theta_2)$ (i.e. $\theta_1,\theta,\theta_2$ are in positive cyclic order). We call $x$ the \emph{root} of $S:=S_f(\theta_1, \theta_2)$, and denote it by $\operatorname{root}(S)$. 

\begin{corollary}
\label{coro-poly-like}
Let $f\in \mathcal P_{d}$ with $|f'(0)|<1$. If $\partial U_f(0)$ contains no critical point and no parabolic point, then there is a polynomial-like restriction $f:U\to V$ with filled Julia set $K(f|_U) = \overline{U_f(0)}$. 
\end{corollary}

\begin{proof}
Without loss of generality, we assume $K(f)$ is connected and $\overline{U_f(0)}\neq K(f)$. 
To make $\overline{U_f(0)}$ into a small filled Julia set, we need to remove the nontrivial limbs of $K(f)$ for $U_f(0)$. 
Let $A$ consist of the roots of all nontrivial limbs. 
By Theorem \ref{RY}(\ref{RY-separate}), for each $y\in A$, there is a unique sector $S_y$ with root $y$ such that $K(f)\cap \overline{S_y} = L_y(K(f),U_f(0))$. 
By Theorem \ref{dynam-prop} and the assumption on $\partial U_f(0)$, every $y\in A$ is pre-repelling. 
Let 
\begin{align*}
C &= \{y\in A\mid S_y\cap\crit(f)\neq\emptyset\},\\
P &= \{f^n(y)\mid y\in C, n\geq1\},\\ 
Q &= f^{-1}(P)\cap \partial U_f(0).
\end{align*}
Then $Q$ is finite and $P\subset Q$.
Given $r>1$. Let $$W_Q = (\mathbb{C}\setminus B_f^{-1}(\mathbb{C}\setminus r\mathbb{D}))\setminus \bigcup_{y\in Q}\overline{S_y}.$$
Replacing $r$ and $Q$ by $r^d$ and $P$ respectively gives $W_P$. 
Note that every point in $Q$ is pre-repelling. By the thickening technique \cite{Mil-lc}, the proper map $f: W_Q\to W_P$ extends to a polynomial-like map $f:U\to V$ such that $\overline{U_f(0)}\subset U$ and $U\cap\crit(f) = \overline{U_f(0)}\cap \crit(f)$. 

We claim that $K(f|_U)=\overline{U_f(0)}$. Clearly, $\overline{U_f(0)}\subset K(f|_U)$. 
By the construction of $f:U \to V$, every point in $S_y$ escapes under $f|_U$ if $y\in Q$. 
Applying Theorem \ref{dynam-prop}(\ref{dynam-prop-crit}) to $y\in A\setminus Q$, there is a minimal $n\geq1$ such that $f^n(y)\in Q$, so the points in $S_y$ also escape under $f|_U$. 
Therefore every point outside $\overline{U_f(0)}$ escapes under $f|_U$. This shows the claim and completes the proof. 
\end{proof}

\begin{corollary}
\label{crit-parab}
For every $f\in \partial_{\rm reg}\mathcal H_d$, the Julia set $J(f)$ is connected, and $\partial U_f(0)$ contains a critical point or a parabolic point. 
\end{corollary}

\begin{proof}
It is clear that $J(f)$ is connected. 
Assume $\partial U_f(0)$ contains no critical point and no parabolic point; we will find a contradiction. By Corollary \ref{coro-poly-like}, there is a polynomial-like map $f:U\to V$ with Julia set $J(f|_U)=\partial U_f(0)$. 
It follows that $\partial U_f(0)$ is a hyperbolic subset of $J(f)$ with $\deg(f|_{\partial U_f(0)})<d$. 
By perturbing $f$ into $\mathcal H_d$, we can get contradictions. 
\end{proof}

Let $$\mathcal{Y}_d^0 = \{f\in\mathcal P_d\mid \text{$|f'(0)|<1$ and $J(f)$ is connected}\}.$$
Then $0$ is attracting and $U_f(0)$ is well-defined for each $f\in\mathcal{Y}_d^0$. 
Let 
$$\mathcal{Y}_d^1=\{f\in\mathcal{Y}_d^0\mid
\text{$\partial U_f(0)$ contains critical points or parabolic points}\}.$$ 
By Corollary \ref{crit-parab}, we have $\partial_{\rm reg}\mathcal H_d\subset \mathcal{Y}_d^1$. 
The conditions in the definition of $\mathcal{Y}_d^1$ allow Fatou components to grow from $U_f(0)$ outward successively.

%
%

\section{Fatou tree of level 1}
\label{sec-tree-level-1}

In this section, we will generate the tree of Fatou components of level 1 for $f\in\mathcal{Y}_d^1$. 
For convenience, we list the notations about the Fatou tree of level $k$ here. The Fatou tree of general level will be defined in \S\ref{subsec-tree-k}. 

\begin{notation}
\label{notation-tree-k}
~

$\mathcal Y_d^k$: the locus consisting of polynomials with Fatou trees of level $k$

$U_1^k(f), U_2^k(f), \dots, U_{n^k(f)}^k(f)$: the union of some parabolic Fatou components

$X_0^k(f)$, $X_1^k(f)$, $X_2^k(f)$, $\dots$: the growing of $k$-tree components

$Y^k(f)$: the Fatou tree of level $k$

$Y_\infty^k(f)$: the set of limit points on $Y^k(f)$

$T^k(f)$: the family of $k$-tree components in $Y^k(f)$

$B^k(f)$: the set of all branches of the tree $(T^k(f),<)$

$\sigma_f^k:B^k(f)\to B^k(f)$: the symbolic map induced by $f$
\end{notation}


%
%
%
%
%
%
%
%
%
%
%
%

%
%
%
%
%
%
%

\subsection{The construction}
\label{level-1-construction}

Given $f\in\mathcal{Y}_d^1$. 
First, we will define $X_0^1(f)$ to be the closure of the union of $U_f(0)$ and attracting petals growing from $U_f(0)$ successively. 
Let $U_0^1(f) = U_f(0)$. 
If $\partial U_0^1(f)$ contains no parabolic periodic point, we set $n^1(f)=0$. 
If $\partial U_0^1(f)$ contains parabolic periodic points, let $U_1^1(f)$ be the union of all parabolic periodic Fatou components $U$ such that $\{f^n|_U\}_{n\in\mathbb{N}}$ tends to a parabolic cycle on $\partial U_0^1(f)$. The boundary of $U_1^1(f)$ possibly contains new parabolic periodic points (outside $\partial U_0^1(f)\cap \partial U_1^1(f)$). Let $U_2^1(f)$ be the union of all parabolic periodic Fatou components $U$ such that $\{f^n|_U\}_{n\in\mathbb{N}}$ tends to a parabolic cycle on $\partial U_1^1(f)\setminus \partial U_0^1(f)$. Inductively, we will get 
$U_0^1(f)$, $U_1^1(f)$, \dots, $U_{n}^1(f)$ such that $\partial U_n^1(f)\setminus \partial U_{n-1}^1(f)$ contains no parabolic periodic point, where $n=n^1(f)\geq1$.  
Now we define $$X_0^1(f)=\bigcup_{j=0}^{n^1(f)}\overline{U_j^1(f)}.$$

\begin{lemma} 
\label{no-rotation}
Let $f\in\mathcal{Y}_d^1$. 
Suppose $z\in X_0^1(f)$ is a parabolic point with period $p$. Then $(f^p)'(z)=1$. 
\end{lemma}

\begin{proof}
If $z\in\partial U_f(0)$, then it follows from $f(U_f(0)) = U_f(0)$ that the combinatorial rotation number of $z$ under $f^p$ is $1$. Therefore $(f^p)'(z)=1$. By induction, we have $(f^p)'(z)=1$ for $z$ in $\partial U_j^1(f)$ successively. 
\end{proof}

Note that the compact set $X_0^1(f)$ satisfies $f(X_0^1(f))=X_0^1(f)$. For any $n\in\mathbb{N}$, define inductively $X_{n+1}^1(f)$ to be the connected component of $f^{-1}(X_n^1(f))$ containing $X_n^1(f)$. See Figure \ref{fig-X1}. Then we have an increasing sequence of connected and compact sets $$X_0^1(f)\subset X_1^1(f)\subset X_2^1(f)\subset\cdots.$$ Each $X_n^1(f)$ is a finite union of closed disks, of which any two are either disjoint or touching at exactly one point on the boundaries. 

\begin{figure}[ht]
\centering
\includegraphics{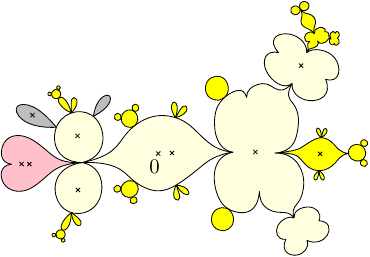}
\caption{The construction of the Fatou tree: from $X_0^1(f)$ (light yellow) to $X_1^1(f)$ (yellow and light yellow). The pink part, which is also a parabolic fixed Fatou component, is not contained in $X_0^1(f)$. The grey parts, which are attracting periodic Fatou components, are not contained in $Y^1(f)$. Each cross is a single critical point.}
\label{fig-X1}
\end{figure}

Let $$Y^1(f)=\overline{\bigcup_{n\in\mathbb{N}} X_n^1(f)}$$ be the \emph{Fatou tree of level $1$}. 
Let $$Y_\infty^1(f)= Y^1(f)\setminus\bigcup_{n\in\mathbb{N}} X_n^1(f)$$ be the set of all limit points on $Y^1(f)$. 

For some small $n$, it may happen that $X_{n+1}^1(f)\subsetneqq f^{-1}(X_n^1(f))\cap Y^1(f)$. 
For example, see the bottom of Figure \ref{fig-critical}.

\subsection{The limb decomposition around the Fatou tree}
\label{subsection-Y-limb}
In set theory, a tree is a partially ordered set (poset) $(T, <)$ such that for each $t \in T$, the set $\{s \in T \mid s < t\}$ is well-ordered by the relation $<$. 

Let $f\in\mathcal{Y}_d^1$. 
The Fatou tree $Y^1(f)$ induces a tree $(T^1(f),<)$ as follows. 
Let $T^1(f)$ be the family of Fatou components in $Y^1(f)$ (or $\bigcup_{n\in\mathbb{N}} X_n^1(f)$). 
For $U\in T^1(f)\setminus\{U_f(0)\}$, there is a unique sector $S_f(U)$ such that $U\subset S_f(U)$, $\operatorname{root}(S_f(U))\in\partial U$, $U_f(0)\subset \mathbb{C}\setminus S_f(U)$ and $S_f(U)$ contains no external ray landing at its root. 
For $U=U_f(0)$, we set $S_f(U)=\mathbb{C}$. 
For $U_1,U_2\in T^1(f)$, we define $U_1<U_2$ if $S_f(U_1)\supsetneqq S_f(U_2)$. 

Let $B^1(f)$ denote the set of all branches of the tree $(T^1(f),<)$. 
Then $K(f)$ has a limb decomposition for $Y^1(f)$: 
\begin{equation}
\label{Kf-Y1f}
K(f) = Y^1(f)^\circ\sqcup \Bigg(\bigsqcup_{y\in \partial Y^1(f)\setminus Y_\infty^1(f)} L_{y}(K(f),Y^1(f))\Bigg) 
\sqcup \Bigg(\bigsqcup_{\mathbf{U}\in B^1(f)} L_{\mathbf{U}}(f)\Bigg),
\end{equation}
where the limbs
\begin{align*}
&L_{y}(K(f),Y^1(f)) = \bigcap_{U\in T^1(f), y\in\partial U} L_{y}(K(f),U),\\
&L_{\mathbf{U}}(f) = K(f)\cap\bigcap_{n\in\mathbb{N}} \overline{S_f(U_n)},\ \mathbf{U}=\{U_n\}_{n\in\mathbb{N}}. 
\end{align*}

\begin{lemma}
\label{Y1-lc}
Let $f\in\mathcal{Y}_d^1$. Suppose $L_{\mathbf{U}}(f)\cap Y^1(f)$ is a singleton for any $\mathbf{U}\in B^1(f)$. Then $Y^1(f)$ is locally connected. 
\end{lemma}

\begin{proof}
Let $y\in Y^1(f)$. We will show that $Y^1(f)$ is locally connected at $y$ by considering three cases. 
The result is clear for $y\in Y^1(f)^\circ$. 
For $y\in \partial Y^1(f)\setminus Y_\infty^1(f)$, the result follows from Theorem \ref{RY}. 
Now consider $y\in Y_\infty^1(f)$. By the limb decomposition of $K(f)$ for $Y^1(f)$, there is a unique $\mathbf{U}=\{U_n\}_{n\in\mathbb{N}}\in B^1(f)$ such that $y\in L_{\mathbf{U}}(f)\cap Y^1(f)$. 
Then $\{y\}= L_{\mathbf{U}}(f)\cap Y^1(f)$ by the hypothesis in the lemma. So $\{S_f(U_n)\cap Y^1(f)\}_{n\in\mathbb{N}}$ is a sequence of open and connected subsets of $Y^1(f)$ decreasing to $y$. 
Hence, $Y^1(f)$ is locally connected at $y\in Y_\infty^1(f)$.
\end{proof}

\begin{remark}
The limit set of a tree of Jordan disks growing as the Fatou tree can be very queer, even not locally connected. 
To show that $L_{\mathbf{U}}(f)\cap Y^1(f)$ is a singleton for any $\mathbf{U}\in B^1(f)$, we should make the most of the dynamics of $f$. 

For example, there might be $y\in Y_\infty^1(f)$ such that $y$ is not an accumulation point of any branch in $B^1(f)$, i.e. $y\notin \bigcap_{n\geq 0}\overline{\bigcup_{k\geq n} U_k}$ for any $\{U_n\}_{n\in\mathbb{N}}\in B^1(f)$. 
To see this, one can cover the locally connected part of a comb by small disks. However, we will exclude such situation for the Fatou tree. 
\end{remark}

The restriction of $f$ on $\bigcup_{n\in\mathbb{N}} X_n^1(f)$ induces a map $\sigma_f^1:B^1(f)\to B^1(f)$ as follows. Let $\mathbf{U}=\{U_n\}_{n\in\mathbb{N}}\in B^1(f)$. 
When $N$ is large enough, the restriction of $f$ on $\bigcup_{n\geq N} \overline{U_n}$ is injective. 
By adding proper elements $V_0,\dots, V_L\in T^1(f)$, we have a branch $\mathbf{V}:=\{V_0,\dots,V_L,f(U_{N}),f(U_{N+1}),\dots\}$, which is independent of $N$. Let $\sigma_f^1(\mathbf{U})=\mathbf{V}$.

\subsection{The preperiodic case}
By Lemma \ref{Y1-lc}, to prove $Y^1(f)$ is locally connected, we just need to show  $L_{\mathbf{U}}(f)\cap Y^1(f)$ is a singleton for any $\mathbf{U}\in B^1(f)$. The preperiodic case will be discussed in this subsection; the wandering case will be proved in the next section. 

%

\begin{lemma}
\label{U-prep} 
Let $f\in\mathcal{Y}_d^1$, and let $\mathbf{U}\in B^1(f)$. 
Suppose $\mathbf{U}$ is preperiodic under $\sigma_f^1$. Then $L_{\mathbf{U}}(f)\cap Y^1(f)$ is a singleton. 
Moreover, if we denote it by $\{y\}$, then $y$ satisfies the following properties.  
\begin{enumerate} 
\item 
\label{U-prep-repel-parab}
$y$ is either pre-repelling or pre-parabolic. 

\item 
\label{U-prep-separate}
When $L_{\mathbf{U}}(f)=\{y\}$, there is only one external ray landing at $y$; when $L_{\mathbf{U}}(f)\neq\{y\}$, two external rays landing at $y$ separate $L_{\mathbf{U}}(f)$ from $Y^1(f)$. 
\end{enumerate}
\end{lemma}

\begin{proof}
Passing to an iterate of $\mathbf{U}$ if necessary, we may
assume that $\mathbf{U}$ is periodic. 
Write $\mathbf{U}=\{U_n\}_{n\in\mathbb{N}}$ and let $p$ denote the period of $\mathbf{U}$. 
The periodicity means there are $N\in\mathbb{N}$ and $k\in\mathbb{Z}$ such that $f^p(U_n) = U_{n-k}$ for any $n\geq N$. Since for any $U\in T^1(f)$, there exists an $n$ such that $f^n(U)\subset X_0^1(f)$, we have $k\geq1$. 

Write $S_f(U_n)=S_f(\theta_n,\theta'_n)$ for $n\geq1$. 
Then $\theta_1,\theta_2,\dots$ (resp. $\theta'_1,\theta'_2,\dots$) are in positive (resp. negative) cyclic order. So they have limits $\theta$ and $\theta'$ respectively. 

For $n$ large enough, $f^p$ is injective on $S_f(U_n)\setminus S_f(U_{n+1})$. 
Combining this with $f^p(U_n) = U_{n-k}$, we have  $d^p\cdot\theta_n \equiv \theta_{n-k}\pmod{\mathbb{Z}}$ and $d^p\cdot\theta'_n \equiv \theta'_{n-k}\pmod{\mathbb{Z}}$.  
Letting $n\to\infty$, we see that $\theta$ and $\theta'$ are fixed under multiplication by $d^p$. In particular, they are rational angles. By \cite[Lemma 3.9]{Kiwi01}, the external rays $R_f(\theta)$ and $R_f(\theta')$ land at a common point, say $y$. Then $y$ is repelling or parabolic. 
If $\theta=\theta'$, then $L_{\mathbf{U}}(f)=\{y\}$; if $\theta\neq \theta'$, then $R_f(\theta)$ and $R_f(\theta')$ separate $L_{\mathbf{U}}(f)$ from $Y^1(f)$. 
In either case, we have $L_{\mathbf{U}}(f)\cap Y^1(f)=\{y\}$. 
\end{proof}

\section{The wandering case}
\label{sect-wandering}

To deal with the wandering case, we need to use the puzzle technique. 

\subsection{Puzzle pieces}
\label{sect-puzzle}

Let $f\in \mathcal Y_d^1$. That is, the boundary $\partial U_f(0)$ contains a critical point or a parabolic point, and $J(f)$ is connected. 
Then we have the Fatou tree $Y^1(f)$ of level $1$. 
A \emph{graph} in $\mathbb{C}$ is a connected set, which can be written as the union of finitely many arcs with pairwise disjoint interiors. In the following, we will associate $f$ with a graph $\Gamma$. 

{\bf Equipotential curves and rays in $\mathbb{C}\setminus K(f)$.}
Recall that $B_f:\mathbb{C}\setminus K(f)\rightarrow\mathbb{C}\setminus \overline{\mathbb{D}}$ is the B\"ottcher map tangent to the identity at $\infty$. 
Let $\Omega_\infty = \{z\in\mathbb{C}\mid |B_f(z)|>r\}$ for some given $r>1$. 

{\bf Equipotential curves and rays in $U_f(0)$ near $\partial U_f(0)$.}
Let $z_0$, $z_1=f(z_0)$, $\dots$, $z_{p_0}=z_0$ be a repelling cycle in $\partial U_f(0)$ with period $p_0\geq2$.  
There are Jordan domains $\Omega'_0\Subset \Omega_0\Subset U_f(0)$ so that $f:\Omega_0\rightarrow \Omega'_0$ is proper and $U_f(0)\cap \crit(f)\subset \Omega_0$. 
For each $0\leq k<p_0$, there exists an open arc $\gamma_k$ in $U_f(0)\setminus \overline{\Omega_0}$ from $\partial \Omega_0$ to $z_k$ such that $\gamma_{k+1} \subset f(\gamma_k) \subset \gamma_{k+1}\cup \overline{\Omega_0}$, where $\gamma_{p_0}=\gamma_0$. 
This is easy to do as follows. Taking a Riemann mapping $h:U_f(0)\to \mathbb{D}$, we get a Blaschke product $\varphi:=h\circ f\circ h^{-1}$ with an attracting fixed point $h(0)$. Then the restriction of $\varphi$ in a neighborhood of the unit circle is a rational-like map, and it is hybrid equivalent to $z^{\deg(\varphi|_{\mathbb{D}})}$. A similar construction in an attracting basin appears in \cite[Lemma 2.1]{KS}. 

%

{\bf Equipotential curves in $U$.}
Let $U\subset X_0^1(f)$ be a parabolic periodic Fatou component with period $p$. 
It is known that there is a Fatou coordinate $\alpha: U\rightarrow \mathbb C$ satisfying $\alpha(f^p(z))=\alpha(z)+1$. 
Let $x\in\mathbb{R}$, and let $E_U=E_U(x)$ be the connected component containing $\lim_{n\in\mathbb{N}}f^{np}|_U$ of the closure of $\alpha^{-1}(\operatorname{Re}(z) = x)$. 
Then $E_U$ is a Jordan curve after requiring $$x\notin\{\Re(\alpha(c))+n\mid c\in U\cap \crit(f^p),n\in\mathbb{Z}\}.$$
Let $\Omega_U$ be the Jordan domain surrounded by $E_U$. By further requiring $$x< \Re(\alpha(c)){\rm\ for\ any\ }c\in U\cap \crit(f^p),$$ we have $\Omega_U\cap(U\cap \crit(f^p))\neq\emptyset$. 
It follows that $\# (E_U\cap \partial U) \geq2$. 
For $1\leq n<p$, let $\Omega_{f^{p-n}(U)} = f^{p-n}(U)\cap f^{-n}(\Omega_U)$. 
Similar constructions in a parabolic basin appear in \cite[\S 3.1.4]{Kiwi04} and \cite[\S 2.1]{RY}. 

{\bf The domain $\Omega$ and the graph $\Gamma$.} 
Let $$\Omega = \mathbb{C} \setminus  \bigcup\left\{\overline{\Omega_\infty},\overline{\Omega_0}, \overline{\Omega_U}\mid \text{parabolic periodic Fatou component $U\subset X_0^1(f)$}\right\}.$$ 
For any repelling or parabolic point in $J(f)$, it is known that there are finitely many external rays landing at it. 
Let $\Gamma$ be the union of $\partial\Omega$, internal rays $\gamma_k$, and tails $R_f(\theta)\setminus \Omega_\infty$ of external rays landing at $z_k$ or points in $E_U\cap \partial U$. See the left in Figure \ref{fig-puzzle}. 
Then $\Gamma$ is a graph, and $f^n(\Gamma)\cap (\Omega\setminus \Gamma) = \emptyset$ for any $n\in\mathbb{N}$. 

\begin{figure}[ht]
\centering
\includegraphics{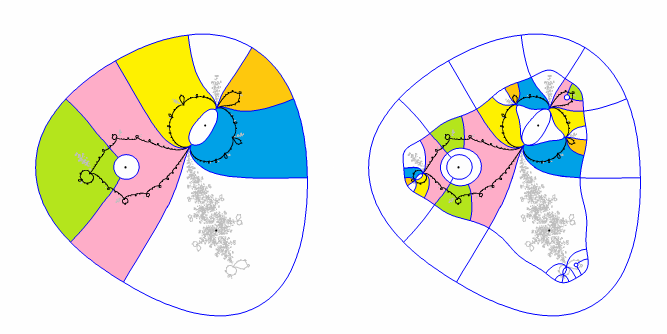}
\caption{Puzzle pieces of depth $0$ and depth $1$. The puzzle pieces intersecting $Y^1(f)$ are shown in colors, that are invariant under $f$.}
\label{fig-puzzle}
\end{figure}

{\bf Puzzle pieces.}
For $n\in\mathbb{N}$, a \emph{puzzle piece} of depth $n$ is a connected component of $f^{-n}(\Omega\setminus \Gamma)$. 
See Figure \ref{fig-puzzle}. Because $\Gamma\cap\partial U_f(0)$ contains a repelling cycle with period at least two and $\# (E_U\cap \partial U) \geq2$, every connected component of $\Omega\setminus\Gamma$ (i.e. a puzzle piece of depth $0$) is a Jordan disk. Because every connected component of the inverse of a Jordan disk under a nonconstant polynomial is still a Jordan disk, every puzzle piece of any depth is a Jordan disk. 

{\bf Marked puzzle pieces.}
A \emph{marked puzzle piece} is a pair $(P_n, z)$ of a puzzle piece $P_n$ of depth $n$ and a point $z\in\partial P_n\cap J(f)$. 
If there is a Fatou component $U$ in $\bigcup_{j\in\mathbb{N}}f^{-j}(X_0^1(f))$ such that $z\in\partial U$ and $P_n\cap U\neq \emptyset$, 
then it is determined by $(P_n,z)$, so we can denote it by $U(P_n,z)$. 
By the construction of the puzzle, $U(P_n,z)$ intersects $\bigcup_{j\in\mathbb{N}} f^{-j}(\Gamma)$. 
According to the attributes of $z$, $\partial P_n$ near $z$ and $U(P_n,z)$, there are six types of $(P_n,z)$ as in Table \ref{six-type}. 
For Types 3 and 6, it may happen that $U(P_n,z)\subset P_n$. 

\begin{table}[ht]
\centering
\includegraphics{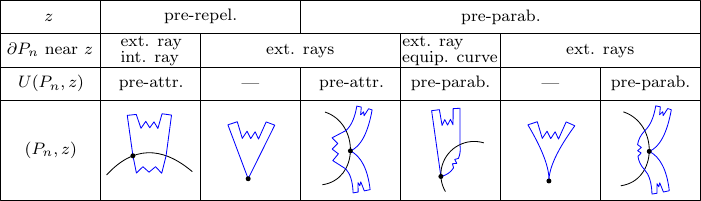}
\caption{Six types of marked puzzle pieces}
\label{six-type}
\end{table}

Since the iterate of $f$ does not change the attributes in Table \ref{six-type}, the marked puzzle pieces $(f(P_n),f(z))$ and $(P_n,z)$ are of the same type when $n\geq1$. Then by observing the puzzle pieces of depth $0$, only Types 1 and 3 or Types 4 and 6 can happen simultaneously for the same puzzle piece. 
If $P_n\cap Y^1(f)\neq \emptyset$, then $(P_n,z)$ is of Type 1, 3, 4, or 6. 

In the following of this section, let $f\in \mathcal Y_d^1$, and construct a puzzle as above. 

%

\subsection{Nests}
A \emph{nest} is a nested sequence $P_0\supset P_1\supset P_2\supset \cdots$ such that $P_n$ is a puzzle piece of depth $n$ for each $n\in\mathbb{N}$. 
We also write $\{P_n\}_{n\in\mathbb{N}}$ by $\mathbf{P}$. 
The \emph{end} of $\mathbf{P}$ is defined by $$\operatorname{end}(\mathbf{P})=\bigcap_{n\in\mathbb{N}}\overline{P_n}.$$ As the intersection of a shrinking sequence of closed Jordan disks, every end is full, connected and compact. 

Let $\mathcal{P}$ denote the collection of all nests. 
Then $\mathcal{P}$ is a topological Cantor set in the natural way, and $f$ induces a map 
$$\begin{array}{llll}
\sigma: &\mathcal{P} &\longrightarrow &\mathcal{P} \\
&\{P_n\}_{n\geq 0}   &\longmapsto &\{f(P_n)\}_{n\geq1}.
\end{array}$$
A nest $\mathbf{P}$ is \emph{preperiodic} if $\sigma^{p+q}(\mathbf{P})=\sigma^q(\mathbf{P})$ for some $p\geq 1,q\geq 0$. The nest is called \emph{periodic} if $q=0$. If there is no such $p,q$, then the nest is called \emph{wandering}. Then $\mathcal{P}$ has a partition $\mathcal{P} = \mathcal{P}_{\rm{p}}\sqcup\mathcal{P}_{\rm{w}}$, where $\mathcal{P}_{\rm{p}}$ and $\mathcal{P}_{\rm{w}}$ consist of preperiodic ones and wandering ones respectively.
Let $$\mathcal{P}^* = 
\{\{P_n\}_{n\in\mathbb{N}} \in\mathcal{P}\mid \text{for any $n$ there is $k\geq1$ so that $P_{n+k} \Subset P_n$}\}.$$


\begin{lemma} 
\label{nest-annulus}
Given $f\in \mathcal Y_d^1$, and construct a puzzle as above. 
Let $\mathbf{P}=\{P_n\}_{n\in\mathbb{N}}$ be a nest. 
Then $\mathbf{P}\in\mathcal{P}^*$ if and only if 
$\operatorname{end}(\mathbf{P}) \cap \bigcup_{n\in\mathbb{N}}f^{-n}(\Gamma)=\emptyset$.
\end{lemma}

\begin{proof}
Assume $\mathbf{P}\in\mathcal{P}^*$. 
Given any $n\in\mathbb{N}$. Choose $k\geq1$ so that $P_{n+k} \Subset P_n$. It follows from $P_n\cap f^{-n}(\Gamma)=\emptyset$ that $\big(\bigcap_{j\geq n+k}\overline{P_j}\big) \cap f^{-n}(\Gamma)=\emptyset$. Then $\operatorname{end}(\mathbf{P})=\bigcap_{j\geq n+k}\overline{P_j}$ and the arbitrariness of $n$ imply $\operatorname{end}(\mathbf{P}) \cap \bigcup_{n\in\mathbb{N}}f^{-n}(\Gamma)=\emptyset$. 

Assume $\mathbf{P}\notin\mathcal{P}^*$. 
There exists an $n\geq 0$ so that $\partial P_n\cap \partial P_{n+k}\neq \emptyset$ for any $k\geq 1$. By the nested property, $\{\partial P_n\cap \partial P_{n+k}\}_{k\geq1}$ is a sequence of decreasing compact sets. 
So we can choose a $z$ in the nonempty set $\bigcap_{k\geq 1}(\partial P_n\cap \partial P_{n+k})$. 
Then $z\in \operatorname{end}(\mathbf{P})$ and $z\in \partial P_n\subset f^{-n}(\Gamma)$ give $\operatorname{end}(\mathbf{P}) \cap \bigcup_{n\in\mathbb{N}}f^{-n}(\Gamma)\neq\emptyset$. 
\end{proof}

\begin{lemma} 
\label{nest-degenerate}
Every nest outside $\mathcal{P}^*$ is preperiodic. That is, $\mathcal{P}\setminus\mathcal{P}^*\subset\mathcal{P}_{\rm p}$.
\end{lemma}

\begin{proof}
Let $\mathbf{P}=\{P_n\}_{n\in\mathbb{N}}\in\mathcal{P}\setminus\mathcal{P}^*$. 
By Lemma \ref{nest-annulus}, we can choose $z\in \operatorname{end}(\mathbf{P}) \cap \bigcup_{n\in\mathbb{N}}f^{-n}(\Gamma)$. 
There is a minimal $n_0$ such that $z\in f^{-n_0}(\Gamma)$. It follows that $z\in\partial P_n$ for any $n\geq n_0$. Because $\partial P_n\subset f^{-n}(\Gamma)$ for $n\in\mathbb{N}$, and $f^{-n}(\Gamma)$ tends to $J(f)$ in Hausdorff distance as $n\to\infty$, we have $z\in J(f)$. 
Then $z\in f^{-n_0}(\Gamma)\cap J(f)$ implies $z$ is preperiodic. Therefore $\mathbf{P}$ is preperiodic. 
\end{proof}

Let $\mathbf{P}$ and $\mathbf{P}^k\ (k\in\mathbb{N})$ be nests. 
Then $\mathbf{P}$ is a limit point of $\mathbf{P}^k$ if and only if for any $n\in\mathbb{N}$, the index set $\{k\in\mathbb{N}\mid P^k_n = P_n\}$ is infinite. 
The \emph{omega limit set} $\omega(\mathbf{P})$ of $\mathbf{P}$ is the set of all limit points of $\sigma^k(\mathbf{P})$ as $k\to\infty$. 
Then $\omega(\sigma(\mathbf{P}))=\omega(\mathbf{P})$, 
$\sigma(\omega(\mathbf{P}))\subset \omega(\mathbf{P})$ and 
$\omega(\mathbf{Q})\subset \omega(\mathbf{P})$ for any $\mathbf{Q}\in \omega(\mathbf{P})$. 
The nest $\mathbf{P}$ is called \emph{recurrent} if $\mathbf{P}\in \omega(\mathbf{P})$. 


\subsection{The bounded degree condition and the elevator condition}
\begin{definition}
[The bounded degree condition and the elevator condition]
A nest $\{P_n\}_{n\in\mathbb{N}}$ is said to satisfy the \emph{bounded degree condition}, if there exists a sequence $\{n_k\}_{k\in\mathbb{N}}$ tending to $\infty$ and a constant $D$ so that $\deg(f^{n_k}|_{P_{n_k}})\leq D$ for any $k\in\mathbb{N}$. Furthermore, if there exist puzzle pieces $Q_N\Subset Q_0$ of depth $N$ and depth $0$ respectively so that $f^{n_k}:(P_{n_k},P_{n_k+N})\to (Q_0,Q_N)$ for any $k\in\mathbb{N}$, then the nest $\{P_n\}_{n\in\mathbb{N}}$ is said to satisfy the \emph{elevator condition}. The triple $(Q_0,Q_N,n_k)$ is called an \emph{elevator} for the nest $\{P_n\}_{n\in\mathbb{N}}$. 
\end{definition}

\begin{lemma}
\label{elevator}
The end of a nest with the elevator condition is a singleton. 
\end{lemma}

\begin{proof}
Let $\mathbf{P}=\{P_n\}_{n\in\mathbb{N}}$ be a nest with the elevator condition as above. Then ${\rm mod}(P_{n_k}\setminus \overline{P_{n_k+N}}) \geq {\rm mod}(Q_0\setminus \overline{Q_N})/D$. After passing to a subsequence, assume $n_{k+1}\geq n_k+N$ for any $k\in\mathbb{N}$. By Gr\"otzsch's inequality, ${\rm mod}(P_0\setminus\operatorname{end}(\mathbf{P}))=\infty$. 
Therefore $\operatorname{end}(\mathbf{P})$ is a singleton. 
\end{proof}

Let $$\mathcal P^\# = \{\{P_n\}_{n\in\mathbb{N}}\in\mathcal P\mid \text{$P_n\cap Y^1(f)\neq \emptyset$ for any $n\in\mathbb{N}$} \}.$$

\begin{lemma} 
\label{BD2elevator}
Every nest in $\mathcal{P}^*\cap \mathcal P^\#$ with the bounded degree condition satisfies the elevator condition. 
\footnote{
The nest in $\mathcal P^\#$ is necessary, since it can happen that a nest in $\mathcal{P}^*\setminus \mathcal P^\#$ with the bounded degree condition does not satisfy the elevator condition. 
}
\end{lemma}

\begin{proof}
Let $\mathbf{P}=\{P_n\}_{n\in\mathbb{N}}\in \mathcal{P}^*\cap \mathcal P^\#$. Suppose $\mathbf{P}$ satisfies the bounded degree condition. 
By definition, there exist $n_k\to \infty$ and $D$ so that $\deg(f^{n_k}|_{P_{n_k}})\leq D$. 
Let $\mathbf{Q}=\{Q_n\}_{n\in\mathbb{N}}$ be a limit point of $\{\sigma^{n_k}(\mathbf{P})\}_{k\in\mathbb{N}}$. 
After passing to a subsequence, we can assume $f^{n_k}(P_{n_k+k}) = Q_k$ for any $k\in\mathbb{N}$. 
 
\vspace{4 pt}
\textbf{Case 1:} $\mathbf{Q}\in \mathcal{P}^*$. 
\vspace{4 pt}
 
In this case, there is $N>0$ so that $Q_N\Subset Q_0$. For $k\geq N$, we have $f^{n_k}:(P_{n_k},P_{n_k+N})\to (Q_0,Q_N)$. The elevator condition follows. 


\vspace{4 pt}
\textbf{Case 2:} $\mathbf{Q}\notin \mathcal{P}^*$. 
\vspace{4 pt}

Note that $\mathcal{P}^\#$ is a closed subset of $\mathcal{P}$. It follows from $\mathbf{P}\in \mathcal{P}^\#$ that $\mathbf{Q}\in \mathcal{P}^\#$. 
We claim that $\operatorname{end}(\mathbf{Q})$ is a singleton. 
By Lemma \ref{nest-annulus}, there is $z\in\operatorname{end}(\mathbf{Q}) \cap f^{-N}(\Gamma)$ for some $N$ large enough. Let $n\geq N$. Then $z\in\partial Q_n\cap J(f)$. 
Since $Q_n\cap Y^1(f)\neq\emptyset$, the marked puzzle piece $(Q_n,z)$ is of Type 1, 3, 4, or 6. So $U(Q_n,z)$ is valid, which is independent of $n\geq N$. By the limb decomposition of $K(f)$ for $U(Q_n,z)$ (Theorem \ref{RY}), we have $\operatorname{end}(\mathbf{Q})=\{z\}$.  

By Lemma \ref{nest-degenerate}, the nest $\mathbf{Q}$ is preperiodic. Let $q\geq0$ be the preperiod and $p\geq1$ be the period of $\mathbf{Q}$. 
Because $\operatorname{end}(\mathbf{Q})$ is a singleton, we can choose $t$ so that $f^j(Q_{q+tp})\cap \crit(f)=\emptyset$ for $0\leq j<q+p$. 
For $k$ large enough, because $f^{n_k}(P_{n_k+k}) = Q_k$, there is a unique $t_k$ such that $f^{n_k}(P_{n_k+q+t_k p})=Q_{q+t_k p}$ and $f^{n_k}(P_{n_k+q+(t_k+1) p})\neq Q_{q+(t_k+1) p}$. Let $n'_k=n_k+q+(t_k-t) p$ and $n''_k=n_k+q+(t_k+1) p$. Then $n_k<n'_k<n''_k$. 
It follows from 
$$P_{n'_k}
\xrightarrow[\deg\leq D]{f^{n_k}}
Q_{q+(t_k-t) p}
\xrightarrow[\deg=1]{f^{q+(t_k-2t)p}}
f^{q}(Q_{q+t p})
\xrightarrow[\deg\leq d^{t p}]{f^{t p}}
f^q(Q_q)$$
that $\deg(f^{n_k'}|_{P_{n_k'}})\leq D':=D d^{t p}$.  See Figure \ref{fig-elevator}. 

\begin{figure}[ht]
\centering
\includegraphics{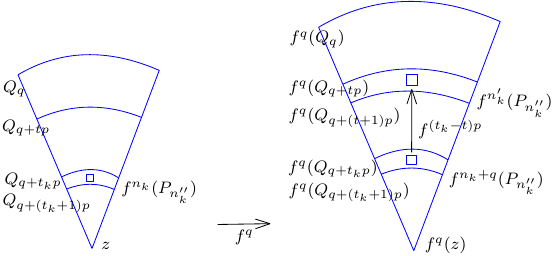}
\caption{A nest $\mathbf{P}\in\mathcal{P}^\#\cap\mathcal{P}^*$ with an omega limit $\mathbf{Q}\in\mathcal{P}^\#\setminus \mathcal{P}^*$}
\label{fig-elevator}
\end{figure}

To complete the proof, we need to show $R_k:=f^{n'_k}(P_{n''_k}) \Subset f^q(Q_q)$. Since $f^q(z)$ is periodic, we have $f^q(z)\in \Gamma\cap J(f)$. So $f^q(z)\in\partial f^q(Q_q)$.  
By modifying $t$, we further require that $\partial f^q(Q_{q+tp})\cap \partial f^q(Q_q)$ is an
arc contained in the closure of one external ray (Type 4), two external rays (Type 3 or 6), or one external ray and one internal ray (Type 1) landing at $f^q(z)$.
It follows from $R_k\subset f^q(Q_{q+tp})\subset f^q(Q_q)$ that $\partial R_k\cap \partial f^q(Q_q)\subset \partial f^q(Q_{q+tp})\cap \partial f^q(Q_q)$. If $\partial R_k\cap \partial f^q(Q_q)$ is nonempty, then it intersects an external ray landing at $f^q(z)$, which implies $f^q(z)\in \partial R_k$. This is impossible, because $R_k\subset f^q(Q_{q+tp}) \setminus f^q(Q_{q+(t+1)p})$. Therefore $\partial R_k\cap \partial f^q(Q_q)=\emptyset$. That is, $R_k\Subset f^q(Q_q)$. 

Since $R_k$ is a puzzle piece of depth $(t+1)p$, after passing to a subsequence, we may assume $R_k$ is a puzzle piece independent of $k$. Then $f^{n_k'}:(P_{n'_k}, P_{n''_k})\to (f^q(Q_q), R_k)$ gives the elevator condition.
\end{proof}

\subsection{Verifications of the bounded degree condition}

\begin{definition}
[The first entry time]
Let $\mathbf{P}=\{P_n\}_{n\in\mathbb{N}}$ be a nest 
and let $Q_k$ be a puzzle piece of depth $k$. The \emph{first entry time} of $\mathbf{P}$ into $Q_k$ is the minimal $n\geq 1$ such that $f^n(P_{n+k})=Q_k$. If no such integer exists, let it be $\infty$. 
\end{definition}

\begin{lemma}
\label{first-entry-time}
Let $r$ be the first entry time of a nest $\mathbf{P}=\{P_n\}_{n\in\mathbb{N}}$ into a puzzle piece $Q_k$ of depth $k$. If $r<\infty$, then the $r$ puzzle pieces $P_{r+k}$, $\dots$, $f^{r-1}(P_{r+k})$ are pairwise disjoint and $\deg(f^r|_{P_{r+k}})\leq d^{d-1}$. 
\end{lemma}

\begin{proof}
Assume $f^{n_1}(P_{r+k})\cap f^{n_2}(P_{r+k})\neq \emptyset$ for some $0\leq n_1<n_2\leq r-1$. Then $f^{n_1}(P_{r+k})\subset f^{n_2}(P_{r+k})$. It follows that $f^{n_1}(P_{r+k-(n_2-n_1)}) = f^{n_2}(P_{r+k})$. So
$f^{r+n_1-n_2}(P_{r+n_1-n_2+k}) 
= f^{r-n_2}(f^{n_2}(P_{r+k}))
= Q_k$. 
This contradicts the minimality of $r$.
Since every critical point of $f$ appears in $P_{r+k}, \dots, f^{r-1}(P_{r+k})$ at most once, we have $\deg(f^r|_{P_{r+k}})\leq d^{d-1}$. 
\end{proof}

\begin{lemma}
\label{BD-1}
Let $\mathbf{P}$ be a wandering nest. If $\omega(\mathbf{P})$ contains a preperiodic nest, then $\mathbf{P}$ satisfies the bounded degree condition. 
\end{lemma}

\begin{proof}
The fact $f(\omega(\mathbf{P}))\subset \omega(\mathbf{P})$ implies that $\omega(\mathbf{P})$ contains at least one periodic nest, say $\mathbf{Q}=\{Q_n\}_{n\in\mathbb{N}}$. Let $p$ be the period of $\mathbf{Q}$. 
Choose $N$ large enough so that $f^k(Q_{k+N}\setminus \operatorname{end}(\mathbf{Q}))$ contains no critical point of $f$ for each $0\leq k<p$. Let $A_k=Q_k\setminus\overline{Q_{k+1}}$ for each $k\in \mathbb{N}$ ($A_k$ may not be an annulus). For any $k\geq N$, any puzzle piece $R$ in $A_k$ will be mapped, by some iterate of $f$, into a puzzle piece in $A_N\cup\cdots\cup A_{N+p-1}$ conformally, because the choice of $N$ guarantees that there is no critical point along the orbit of $R$. 

For each $k\geq N$, let $r_k\geq 1$ be the first entry time of $\mathbf{P}$ into $Q_k$. By Lemma \ref{first-entry-time}, the degree of $f^{r_k}:P_{r_k+k}\to Q_k$ is at most $d^{d-1}$. Since $\sigma^{r_k}(\mathbf{P})\neq\mathbf{Q}$, there is a unique $s_k\geq 0$ so that $f^{r_k}(P_{r_k+k+s_k+1})\subset A_{k+s_k}$. Let $t_k$ be the integer so that $N\leq k+s_k-t_k\leq N+p-1$ and $p\mid t_k$, and let $n_k = r_k+k+s_k+1$. Then 
$$ P_{n_k} 
\xrightarrow[\deg\leq d^{d-1}]{f^{r_k}} f^{r_k}(P_{n_k})
\xrightarrow[\deg=1]{f^{t_k}} f^{r_k+t_k}(P_{n_k})
\xrightarrow[\deg\leq d^{N+p}]{f^{n_k-r_k-t_k}} f^{n_k}(P_{n_k}).$$ Therefore $\mathbf{P}$ satisfies the bounded degree condition. 
\end{proof}

\begin{lemma}
\label{BD-2}
Let $\mathbf{P}$ be a wandering nest. If there is $\mathbf{Q}\in\omega(\mathbf{P})$ such that $\omega(\mathbf{Q})\subsetneqq \omega(\mathbf{P})$, then $\mathbf{P}$ satisfies the bounded degree condition. 
\end{lemma}

\begin{proof}
Let $\mathbf{R}\in \omega(\mathbf{P})\setminus \omega(\mathbf{Q})$. 
By $\mathbf{R}\notin \omega(\mathbf{Q})$, we can choose $L$ and $N$ so that $f^k(Q_{k+L})\cap R_L=\emptyset$ for any $k\geq N$. Let $\mathbf{Q}'=\sigma^N(\mathbf{Q})$. Then $f^k(Q'_{k+L})\cap R_L=\emptyset$ for any $k\geq 0$. 
Now fix $k\geq 0$. 
Since $\mathbf{Q}'\in\omega(\mathbf{P})$, there is a minimal $r_k\geq 1$ such that $f^{r_k}(P_{r_k+k})=Q'_k$. 
Since $\mathbf{R}\in\omega(\sigma^{r_k}(\mathbf{P}))$, there is a minimal $s_k\geq 1$ such that $f^{s_k}(f^{r_k}(P_{r_k+s_k+L}))=R_L$. Let $S = f^{r_k}(P_{r_k+s_k+L})$, which is a puzzle piece of depth $s_k+L$.  
By Lemma \ref{first-entry-time}, both $\deg(f^{r_k}|_{P_{r_k+k}})$ and $\deg(f^{s_k}|_S)$ are bounded above by $d^{d-1}$. See Figure \ref{fig-first-entries}. 

\begin{figure}[ht]
\centering
\includegraphics{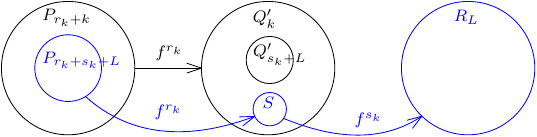}
\caption{First entries of $\mathbf{P}$ into $Q'_k$ (black) and $\sigma^{r_k}(\mathbf{P})$ into $R_L$ (blue)}
\label{fig-first-entries}
\end{figure}
 
Since $P_{r_k+s_k+L}\cap P_{r_k+k}\neq\emptyset$, we have  $S\cap Q'_k\neq\emptyset$. 
If $Q'_k\subset S$, then $S=Q'_{s_k+L}$, which would imply $f^{s_k}(Q'_{s_k+L}) = R_L$. This is a contradiction. Hence, $S\subsetneqq Q'_k$ and $S\cap Q'_{s_k+L}=\emptyset$. 
Then $\deg(f^{r_k}|_{P_{r_k+s_k+L} })\leq \deg(f^{r_k}|_{P_{r_k+k}})\leq d^{d-1}$. 
Let $n_k = r_k+s_k+L$. Since $s_k+L> k$, we have $n_k\to\infty$ as $k\to\infty$. 
Finally, $\deg(f^{n_k}:P_{n_k}\to f^L(R_L))\leq d^{2(d-1)+L}$ implies that $\mathbf{P}$ has the bounded degree condition.
\end{proof}


Let
$$\mathcal{P}_{\rm c} = \{\{P_n\}_{n\in\mathbb{N}}\in\mathcal{P}\mid \text{$P_n\cap \crit(f)\neq\emptyset$ for any $n\in\mathbb{N}$}\}.$$
Each nest in $\mathcal{P}_{\rm c}$ is called \emph{critical}. 

\begin{lemma}
\label{BD-non-critical}
Let $\mathbf{P}$ be a wandering nest. 
If $\omega(\mathbf{P})$ contains no critical nest, 
then $\mathbf{P}$ satisfies the bounded degree condition.  
\end{lemma}

\begin{proof} 
By $\omega(\mathbf{P})\cap \mathcal{P}_{\rm c}=\emptyset$, 
there are $L,N\in\mathbb{N}$ such that $f^n(P_{n+L})\cap C_L=\emptyset$ for any $n\geq N$ and any $\mathbf{C}\in\mathcal P_{\rm c}$. 
Increasing the depth $L$ if necessary, we may assume every puzzle piece of depth $L$ outside $\bigcup_{\mathbf{C}\in\mathcal P_{\rm c}}C_L$ contains no critical point. 
Then for each $n\geq N$, we have 
$$ P_{n+L} 
\xrightarrow[\deg\leq d^N]{f^N} f^N(P_{n+L})
\xrightarrow[\deg=1]{f^{n-N}} f^n(P_{n+L})
\xrightarrow[\deg\leq d^L]{f^L} f^{n+L}(P_{n+L})
.$$ 
Here the middle part follows from that $f^{N+k}(P_{n+L})$ contains no critical point for each $0\leq k\leq n-N$, since $f^{N+k}(P_{n+L})\subset f^{N+k}(P_{N+k+L})$. 
Therefore $\mathbf{P}$ satisfies the bounded degree condition. 
\end{proof}


\begin{assumption}
\label{asp-n0}
Take $n_0$ large enough satisfying the following properties. 
\begin{enumerate}
\item $f^{-n_0}(\Omega\setminus\Gamma)\cap\crit(f)\subset \bigcup_{\mathbf{C}\in\mathcal P_{\rm c}}\operatorname{end}(\mathbf{C})$, where $f^{-n_0}(\Omega\setminus\Gamma)$ is the union of all puzzle pieces of depth $n_0$. 

\item For any different critical nest $\mathbf{C}$ and $ \mathbf{C}'$, we have $C_{n_0}\cap C'_{n_0} = \emptyset$. 

\item Let $\mathbf{C},\mathbf{C}'\in\mathcal{P}_{\rm c}$. If $\mathbf{C}'\notin \omega(\mathbf{C})$, then $f^k(C_{k+n_0})\cap  C'_{n_0}=\emptyset$ for any $k\geq0$ except at most one $k$ with $\sigma^k(\mathbf{C}) = \mathbf{C}'$.

\item For any $\mathbf{C}\in\mathcal{P}_{\rm c}\cap \mathcal P^*$, we have $C_{n_0}\Subset C_0$.  
\end{enumerate}
\end{assumption}

Let $\mathbf{C}$ be a critical nest. For $n\geq n_0$ and $k\geq1$, we call $C_{n+k}$ a \emph{successor} of $C_n$ if $f^k(C_{n+k})=C_n$ and each critical nest appears at most twice along the orbit $C_{n+k}, f(C_{n+k}),\dots, f^k(C_{n+k})=C_n$. It is clear that the degree of $f^k:C_{n+k}\to C_n$ is at most $d^{2d-1}$. 

\begin{lemma}
\label{BD-reluctant}
Let $\mathbf{P}$ be a wandering nest. 
If there is a critical nest $\mathbf{C}$ in $\omega(\mathbf{P})$ and an $n\geq n_0$ such that $C_n$ has infinitely many successors, then $\mathbf{P}$ satisfies the bounded degree condition. 
\end{lemma}

\begin{proof} 
Let $\{C_{r_k}\}_{k\geq 1}$ be all successors of $C_n$ with $n<r_1<r_2<\cdots$. Since $\mathbf{C}\in \omega(\mathbf{P})$, there is a minimal $s_k\geq 1$ such that $f^{s_k}(P_{s_k+r_k}) = C_{r_k}$ for each $k\geq 1$. 
Then 
$$ P_{s_k+r_k} 
\xrightarrow[\deg\leq d^{d-1}]{f^{s_k}} C_{r_k}
\xrightarrow[\deg\leq d^{2d-1}]{f^{r_k-n}} C_n
\xrightarrow[\deg\leq d^n]{f^n} f^n(C_n)
.$$ 
Here the first inequality follows from Lemma \ref{first-entry-time}, and the middle one follows from the definition of successors. This shows $\mathbf{P}$ satisfies the bounded degree condition. 
\end{proof}

\subsection{The persistent recurrence case}
Contrary to the conditions in Lemmas \ref{BD-1}, \ref{BD-2}, \ref{BD-non-critical} and \ref{BD-reluctant}, we have the following condition. 

\begin{definition}
A wandering nest $\mathbf{P}$ is said to satisfy the \emph{persistent recurrence condition} if the following properties hold. 
\begin{itemize}
\item $\omega(\mathbf{P})\subset\mathcal P_{\rm w}$. 

\item $\omega(\mathbf{P})\cap \mathcal P_{\rm c}\neq\emptyset$. 

\item $\mathbf{C}'\in\omega(\mathbf{C})$ and $\mathbf{C}\in\omega(\mathbf{C}')$ for any $\mathbf{C}, \mathbf{C}'\in \omega(\mathbf{P})\cap \mathcal P_{\rm c}$. 

\item $C_n$ has only finitely many successors for any $\mathbf{C}\in \omega(\mathbf{P})\cap \mathcal P_{\rm c}$ and $n\geq n_0$. 
\end{itemize}
Each critical nest in $\omega(\mathbf{P})\cap \mathcal P_{\rm c}$ is called \emph{persistently recurrent}. 
\end{definition}

Let $\mathbf{C}$ be a persistently recurrent critical nest. Starting with $C_{n_0}$, we can inductively define $C_{\widetilde n_j} = \mathcal{A}\mathcal{A}\boldsymbol{\Gamma}^\tau(C_{n_{j-1}})$ and $C_{n_j} = \mathcal{B}\mathcal{A}\boldsymbol{\Gamma}^\tau(C_{n_{j-1}})$, where $\mathcal{A},\mathcal{B}, \boldsymbol{\Gamma}$ are operators producing deeper critical puzzle pieces introduced in \cite{KSS}, and $\tau$ is a constant. Let $k_j>0$ denote the depth gap between $ \mathcal{A}\boldsymbol{\Gamma}^\tau (C_{n_{j-1}})$ and $\mathcal{B}\boldsymbol{\Gamma}^\tau(C_{n_{j-1}})$. Let $n'_j = n_j-k_j$. 
We call such subnest of $\mathbf{C}$ a \emph{principal nest}, whose significant properties are summarized as follows. 

\begin{theorem} 
\label{p-nest} 
Let $\mathbf{P}$ be a wandering nest with the persistent recurrence condition, and let $\mathbf{C}\in \omega(\mathbf{P})\cap \mathcal P_{\rm c}$. 
Then there are constants $D\geq2$, $m>0$ and $j_0\geq1$ independent of $\mathbf{P}$ and $\mathbf{C}$ such that the subnest 
$$C_{n_0}
\supset
C_{n'_1}\supset C_{n_1}\supset C_{\widetilde n_1}
\supset
C_{n'_2}\supset C_{n_2}\supset C_{\widetilde n_2}
\supset
\cdots$$ 
satisfies the following properties. 
\begin{enumerate}
\item For any $j\geq1$,
$$\deg(f^{n'_j-n_{j-1}}|_{C_{n'_j}}),\ 
\deg(f^{n_j-n_{j-1}}|_{C_{n_j}}),\ 
\deg(f^{\widetilde n_j-n_{j-1}}|_{C_{\widetilde n_j}})
\leq D.$$
 
\item 
\label{gap}
The gaps $\widetilde n_j - n_j$ and $n_j-n'_j$ both tend to infinity as $j\to\infty$. 

\item 
\label{no-postcrit}
For all $j\geq 1$, $\mathbf{C}'\in\omega(\mathbf{P})\cap \mathcal P_{\rm c}$ and $k\in\mathbb{N}$, we have $$(C_{n'_j}\setminus \overline{C_{\widetilde n_j}} )\cap f^k(C'_{k+\widetilde n_j})=\emptyset.$$ 
 
\item 
\label{complex-bound}
(Complex bound)
For each $j\geq j_0$, 
$$\operatorname{mod}(C_{n'_j}\setminus \overline{C_{n_j}}),\ 
\operatorname{mod}(C_{n_j}\setminus \overline{C_{\widetilde n_j}})\geq m.$$
\end{enumerate}
\end{theorem}

The construction of the principal nest is attributed to Kahn-Lyubich \cite{KL-lc} in the unicritical case, Kozlovski-Shen-van Strien \cite{KSS} in the multicritical case. The complex bounds are proven by Kahn-Lyubich \cite{KL-lc, KL-lem} (unicritical case), Kozlovski-van Strien \cite{KS} and Qiu-Yin \cite{QY} independently (multicritical case). 
The construction of $C_{\widetilde n_j}$ and the estimate of $\operatorname{mod}(C_{n_j}\setminus \overline{C_{\widetilde n_j}})$ appeared in \cite{YZ} for the first time; see also \cite[\S4.6]{PQRTY}. 
The interested readers may see these references for a detailed construction of the nest and the proof of its properties. 

We remark that in our setting, the annuli $C_{n'_j}\setminus \overline{C_{n_j}}$ might be degenerate for the first few $j$'s. But because of $C_{n_0}\Subset C_0$ and the growth of the gaps $n_j-n'_j$, the annuli $C_{n'_j}\setminus \overline{C_{n_j}}$ will be nondegenerate when $j$ is large enough. 

\begin{lemma}
\label{persistent-end}
Let $\mathbf{P}$ be a wandering nest. 
If $\mathbf{P}$ satisfies the persistent recurrence condition, then $\operatorname{end}(\mathbf{P})$ is a singleton. 
\end{lemma}

\begin{proof}
Let $\mathbf{C}\in \omega(\mathbf{P})\cap \mathcal P_{\rm c}$ and construct a principle nest as Theorem \ref{p-nest}. 
Choose $j_1=j_1(\mathbf{P})\geq j_0$ such that for each $\mathbf{C}''\in \mathcal P_{\rm c}\setminus \omega(\mathbf{P})$, we have  $$f^k(P_{k+n'_{j_1}})\cap  C''_{n'_{j_1}}=\emptyset$$ for any $k\geq0$ except at most one $k$ with $\sigma^k(\mathbf{P}) = \mathbf{C}''$. 
For $j\geq j_1$, let $r_j\geq 1$ be the first entry time of $\mathbf{P}$ into $C_{n_j}$. By Theorem \ref{p-nest}(\ref{no-postcrit}), 
$$\deg(f^{r_j}:P_{r_j+n'_j}\to C_{n'_j})=\deg(f^{r_j}:P_{r_j+n_j}\to C_{n_j}).$$ 
By Lemma \ref{first-entry-time}, this degree is bounded by $d^{d-1}$. 
By Theorem \ref{p-nest}(\ref{complex-bound}), 
$$\operatorname{mod}(P_{r_j+n'_j}\setminus\overline{P_{r_j+n_j}})\geq \frac{\operatorname{mod}(C_{n'_j}\setminus\overline{C_{n_j}})}{d^{d-1}}\geq \frac{m}{d^{d-1}}.$$ 
Since $r_{j+1}\geq r_j$, the annuli $\{P_{r_j+n'_j}\setminus\overline{P_{r_j+n_j}}\}_{j\geq j_1}$ are nested. 
Then the Gr\"otzsch inequality implies $\operatorname{mod}(P_0\setminus\operatorname{end}(\mathbf{P}))=\infty$. 
So $\operatorname{end}(\mathbf{P})$ is a singleton.
\end{proof}


\begin{proposition}
\label{wandering-end}
Let $\mathbf{P}$ be a wandering nest intersecting $Y^1(f)$, i.e. $\mathbf{P}\in\mathcal P_{\rm w}\cap \mathcal{P}^\#$. Then $\operatorname{end}(\mathbf{P})$ is a singleton. 
\end{proposition}

\begin{proof}
If $\mathbf{P}$ satisfies the persistent recurrence condition, then $\operatorname{end}(\mathbf{P})$ is a singleton by Lemma \ref{persistent-end}. 
Now assume $\mathbf{P}$ does not satisfy the persistent recurrence condition. By Lemmas \ref{BD-1} to \ref{BD-reluctant}, the nest $\mathbf{P}$ satisfies the bounded degree condition. By Lemma \ref{nest-degenerate}, we have $\mathbf{P}\in \mathcal P_{\rm w}\subset\mathcal P^*$. Then Lemma \ref{BD2elevator} implies that $\mathbf{P}$ satisfies the elevator condition. It follows from Lemma \ref{elevator} that $\operatorname{end}(\mathbf{P})$ is a singleton. 
\end{proof}

\section{Local connectivity}
\label{sec-lc}

Now we return to the discussion in \S\ref{sec-tree-level-1}. 
Let $f\in \mathcal Y_d^1$. Then the boundary $\partial U_f(0)$ contains a critical point or a parabolic point, and $J(f)$ is connected. 
This guarantees that the Fatou tree $Y^1(f)$ of level $1$ is well-defined, and the filled Julia set $K(f)$ has a limb decomposition for $Y^1(f)$. 
Up to now, we haven't fully understand the connection between $Y^1(f)$ and a limb growing from $Y_\infty^1(f)$.  

Recall that each $\mathbf{U}\in B^1(f)$ determines a decreasing sequence $\{S_f(U_n)\}_{n\in\mathbb{N}}$ of sectors, and the limb $L_{\mathbf{U}}(f)$ is the intersection of $K(f)$ and these sectors. 
Lemma \ref{Y1-lc} says that if $L_{\mathbf{U}}(f)\cap Y^1(f)$ is a singleton for any $\mathbf{U}\in B^1(f)$, then $Y^1(f)$ is locally connected. 
By Lemma \ref{U-prep}, if $\mathbf{U}$ is preperiodic, then $L_{\mathbf{U}}(f)\cap Y^1(f)$ is a singleton. 
To prove the local connectivity of $Y^1(f)$, it needs the last puzzle piece. 

\begin{lemma}
\label{U-wand} 
Let $f\in\mathcal{Y}_d^1$, and let $\mathbf{U}\in B^1(f)$. 
Suppose $\mathbf{U}$ is wandering under $\sigma_f^1$. Then $L_{\mathbf{U}}(f)\cap Y^1(f)$ is a singleton. 
Moreover, if we denote it by $\{y\}$, then $y$ satisfies the following properties.  
\begin{enumerate} 
\item $y$ is wandering. 

\item 
\label{U-wand-separate}
When $L_{\mathbf{U}}(f)=\{y\}$, there is only one external ray landing at $y$; when $L_{\mathbf{U}}(f)\neq\{y\}$, two external rays landing at $y$ separate $L_{\mathbf{U}}(f)$ from $Y^1(f)$. 
\end{enumerate}
\end{lemma}

\begin{proof} 
We claim that $X:= L_{\mathbf{U}}(f)\cap Y^1(f)$ is a wandering continuum. Since $X = \bigcap_{n\in\mathbb{N}}(\overline{S_f(U_n)}\cap Y^1(f))$, it is connected. Given $N>0$. Since $\mathbf{U}$ is wandering, there is an $M$ such that the sectors in $\{S_f(f^n(U_M))\}_{0\leq n\leq N}$ are pairwise disjoint. Then $\{f^n(X)\}_{0\leq n\leq N}$ are pairwise disjoint. Because $N$ is arbitrary, the continuum $X$ is wandering. 

Construct a puzzle as in \S\ref{sect-wandering}. 
Recall that $\mathcal{P}$ is the collection of all nests. 
and $f$ induces a map $\sigma: \mathcal{P}\to\mathcal{P},
\{P_n\}_{n\geq 0}\mapsto\{f(P_n)\}_{n\geq1}$. 
Since $X$ is a wandering continuum, there is a unique wandering nest $\mathbf{P}$ such that $X\subset\operatorname{end}(\mathbf{P})$. 
By Proposition \ref{wandering-end}, the end of $\mathbf{P}$ is a singleton, and so is $X$. 

Let $X=\{y\}$. Then $y$ is a wandering point. To construct external rays landing at $y$, write $S_f(U_n)=S_f(\theta_n,\theta'_n)$ for $n\geq1$. 
Then $\theta_1,\theta_2,\dots$ (resp. $\theta'_1,\theta'_2,\dots$) are in positive (resp. negative) cyclic order. So they have limits $\theta$ and $\theta'$ respectively. For any $k\geq0$, the external rays $R_f(\theta)$ and $R_f(\theta')$ will enter into $P_k$ and stay in it. Since the puzzle piece $P_k$ shrinks down to the point $y$ as $k\to \infty$, these two external rays land at $y$. If $\theta=\theta'$, then $L_{\mathbf{U}}(f)=\{y\}$; if $\theta\neq \theta'$, then $R_f(\theta)$ and $R_f(\theta')$ separate $L_{\mathbf{U}}(f)$ from $Y^1(f)$. 
\end{proof}

\begin{theorem}
\label{thm-Y1-lc}
Let $f\in\mathcal{Y}_d^1$. Then $Y^1(f)$ is locally connected. 
\end{theorem}

\begin{proof}
By Lemmas \ref{U-prep} and \ref{U-wand}, for any $\mathbf{U}\in B^1(f)$, the intersection of $L_{\mathbf{U}}(f)$ and $Y^1(f)$ is a singleton. 
Then Lemma \ref{Y1-lc} gives the conclusion. 
\end{proof}

\begin{corollary}
[Limb decomposition]
\label{coro-limb}
Let $f\in\mathcal{Y}_d^1$. Then the following properties hold. 
\begin{enumerate}
\item There is a unique decomposition $$K(f)=Y^1(f)^\circ\sqcup \bigsqcup_{y\in \partial Y^1(f)} L_y$$ such that $L_y\cap Y^1(f)=\{y\}$ and $L_y$ is a continuum for any $y\in \partial Y^1(f)$. 

\item 
\label{coro-limb-ray}
For any $y\in \partial Y^1(f)$, there exist external rays landing at $y$. If $L_y\neq \{y\}$, there are finitely many sectors with root $y$ separating $L_y$ from $Y^1(f)$. 

\item The diameters of $L_y$'s tend to zero. 
\end{enumerate}
\end{corollary}

\begin{proof}
For $y\in\partial Y^1(f)\setminus \partial Y^1_\infty(f)$, we let $L_y = L_y(K(f),Y^1(f))$ as in \eqref{Kf-Y1f}. 
Now consider $y\in Y^1_\infty(f)$. There is a unique $\mathbf{U}\in B^1(f)$ such that $y\in L_{\mathbf{U}}(f)$.  
Define $L_y= L_{\mathbf{U}}(f)$. By Lemmas \ref{U-prep} and \ref{U-wand}, we have $L_{\mathbf{U}}(f)\cap Y^1(f)=\{y\}$. Therefore the sets in $\{L_y\}_{y\in Y^1_\infty(f)}$ are pairwise disjoint. 
This gives the desired decomposition, which is the same as \eqref{Kf-Y1f}. 

By Theorem \ref{RY}(\ref{RY-separate}), Lemmas \ref{U-prep}(\ref{U-prep-separate}) and \ref{U-wand}(\ref{U-wand-separate}), 
for any trivial limb $L_y$, there is at least one external ray landing at $y$; 
for any nontrivial limb $L_y$, there are finitely many sectors $S_1,\dots, S_n$ with root $y$ such that $L_y=K(f)\cap\bigcup_{1\leq k\leq n} \overline{S_k}$. 
Furthermore, if we require that $\overline{S_k}\setminus\{y\}$ with $1\leq k\leq n$ are pairwise disjoint, then the family $\{S_k\}_{1\leq k\leq n}$ of sectors is unique.  
This gives the uniqueness of the decomposition. 


Let $\phi:\mathbb{C}\setminus Y^1(f)\to \mathbb{C}\setminus \overline{\mathbb{D}}$ be a conformal map. 
By Theorem \ref{thm-Y1-lc}, $\phi^{-1}$ can be extended to a continuous map $\psi:\mathbb{C}\setminus \mathbb{D}\to \mathbb{C}\setminus Y^1(f)^{\circ}$. 
Let $y\in \partial Y^1(f)$. 
If $L_y$ is nontrivial, let $\mathcal{A}_y$ be the family $\{S_k\}_{1\leq k\leq n}$ of sectors as above. Consider $S = S_f(\theta,\theta')\in\mathcal{A}_y$. By \cite[Corollary 6.4]{McM}, $\phi(R_f(\theta))$ and $\phi(R_f(\theta'))$ land at points in $\partial\mathbb{D}$. Since $S\cap Y^1(f)=\emptyset$, they lands at the same point. 
If $L_y$ is trivial, let $\mathcal{A}_y$ be the family of all external rays landing at $y$. Then for each $R_f(\theta)\in \mathcal{A}_y$, by \cite[Corollary 6.4]{McM} again, the ray $\phi(R_f(\theta))$ lands at a point in $\partial\mathbb{D}$. 
By $$\mathbb{C} = \overline{\mathbb{D}}\sqcup\bigsqcup_{
y\in \partial Y^1(f),\ S\in\mathcal{A}_y} \phi(S),$$
the limb decomposition of $\mathbb{C}\setminus \phi(\mathbb{C}\setminus K(f))$ for $\mathbb{D}$ satisfies the first condition in Lemma \ref{iff-diam-0}, which implies that the diameters of $\phi(S\cap K(f))$'s tend to zero. By the uniform continuity of $\psi$, the diameters of $S\cap K(f)$'s tend to zero. It follows that the diameters of $L_y$'s tend to zero. 
\end{proof}


We call $L_y$ the \emph{limb} for $(K(f),Y^1(f))$ with root $y$. We also denote it by $L_y(K(f),Y^1(f))$ to distinguish the limbs of $K(f)$ for a Fatou component. 

Let $y\in \partial Y^1(f)$. By Corollary \ref{coro-limb}(\ref{coro-limb-ray}), there is at least one external ray landing at $y$. 
Then by \cite[Corollary 6.7]{McM}, the number of the connected components of $K(f)\setminus\{y\}$ is exactly the number of the external rays landing at $y$. If $y$ is wandering, then we can compute this number. 

\begin{corollary} 
[Number of external rays]
\label{ray-num}
Let $f\in\mathcal{Y}_d^1$. 
If $y\in \partial Y^1(f)$ is wandering, then the number $N(y)$ of external rays landing at $y$ is given by 
$$N(y)=\prod_{n\in\mathbb{N}}\deg(f, f^n(y)).$$
Moreover, we have the following dichotomy: 
\begin{enumerate}
\item If $y\in \partial Y^1(f)\setminus Y^1_\infty(f)$ is wandering, then the number of connected components of $Y^1(f)\setminus\{y\}$ is $N(y)$. 

\item If $y\in Y^1_\infty(f)$ is wandering, then $Y^1(f)\setminus\{y\}$ is connected. 
\end{enumerate}
\end{corollary}

\begin{proof} 
Let $y$ be a wandering point in $\partial Y^1(f)$. 
By Corollary \ref{coro-limb}(\ref{coro-limb-ray}), we have $N(y)\geq 1$. 
We claim that $N(y)=1$ if $\crit_y(f)=\emptyset$. 
Assume $N(y)>1$ and $\crit_y(f)=\emptyset$; we will find a contradiction. Since the orbit of $y$ avoids critical points, we have $N(f^n(y))=N(y)$ for any $n\geq 0$. 
Consider two cases. 

The first case is $y\in \partial Y^1(f)\setminus Y^1_\infty(f)$. In this case, there is a $q$ large enough such that $f^q(y)$ is on the boundary of a periodic attracting or parabolic Fatou component $U$, whose period is denoted by $p$. Because of $N(f^q(y))=N(y)>1$, we can choose a sector $S_f(\alpha,\beta)$ with root $f^q(y)$ and disjoint with $U$. Since $y$ is wandering, the sectors $S_f(d^{np}\alpha, d^{np}\beta)$ $(n\geq0)$ are pairwise disjoint. When $n$ is large enough, the sector $S_f(d^{np}\alpha, d^{np}\beta)$ contains no critical point, so the angular difference of $d^{np}\alpha$ and $d^{np}\beta$ grows exponentially as $n\to\infty$. This is a contradiction. 

The second case is $y\in Y^1_\infty(f)$. 
Since $N(y)>1$, there is a sector $S_f(\alpha,\beta)$ with root $y$ and disjoint with $Y^1(f)$. 
Similar to the first case, we can deduce a contradiction through $S_f(d^n\alpha, d^n\beta)$ $(n\geq0)$. 

Therefore $N(y)=1$ if $\crit_y(f)=\emptyset$. If $\crit_y(f)\neq \emptyset$, it concludes by 
$$N(f^n(y))=N(f^{n+1}(y))\cdot \deg(f, f^n(y)).$$

Assume $y\in \partial Y^1(f)\setminus Y^1_\infty(f)$ is wandering. By the construction of $Y^1(f)$, the number of connected components of $Y^1(f)\setminus\{y\}$ is also $\prod_{n\in\mathbb{N}}\deg(f, f^n(y))$.  

Assume $y\in Y^1_\infty(f)$ is wandering. 
There is a unique $\mathbf{U}\in B^1(f)$ such that $y\in L_{\mathbf{U}}(f)$.  
By Lemmas \ref{U-prep} and \ref{U-wand}, we have $L_{\mathbf{U}}(f)\cap Y^1(f)=\{y\}$. 
So $Y^1(f)\setminus\{y\}$ is equal to $\bigcup_{n\in\mathbb{N}}(Y^1(f)\setminus S_f(U_n))$, which is connected. 
\end{proof}

%
%

\section{Fatou tree of general level}
\label{sec-tree-k}

In this section, we will construct the Fatou tree of higher level, including the maximal Fatou tree (i.e. the Fatou tree of highest level). Then we will give a necessary but not sufficient condition for a polynomial on the regular boundary of the central hyperbolic component. 
Finally, we extend Fatou trees to $f\in\mathcal P_d$ with a parabolic fixed point at $0$. 

\subsection{Fatou tree of level $k$}
\label{subsec-tree-k}

Recall that $$\mathcal{Y}_d^0 = \{f\in\mathcal P_d\mid \text{$|f'(0)|<1$ and $J(f)$ is connected}\}.$$
To unify the form of notations, for each $f\in \mathcal{Y}_d^0$, let $Y^0(f) = \overline{U_f(0)}$ be the \emph{Fatou tree of level $0$}, and let $Y^0_\infty(f) = \partial U_f(0)$. 

Recall also that 
$$\mathcal{Y}_d^1=\{f\in\mathcal{Y}_d^0\mid
\text{$Y^0_\infty(f)$ contains critical points or parabolic points}\}.$$ 
For each $f\in\mathcal{Y}_d^1$, we have constructed the Fatou tree $Y^1(f)$ of level $1$ and its limit boundary $Y^1_\infty(f)$. 

For each $k\geq1$, we can inductively construct the Fatou tree $Y^k(f)$ of level $k$ from $Y^{k-1}(f)$ similar to the construction of $Y^1(f)$ from $Y^0(f)$.  We repeat the procedure as follows. First, let $$\mathcal{Y}_d^k=\{f\in\mathcal{Y}_d^{k-1}\mid
\text{$Y^{k-1}_\infty(f)$ contains critical points or parabolic points}\}.$$ 

Let $f\in\mathcal{Y}_d^k$, and let $U^k_0(f) = Y^{k-1}(f)\setminus Y^{k-1}_\infty(f)$. 
If $Y^{k-1}_\infty(f)$ contains no parabolic point, set $n^k(f)=0$; otherwise, there are parabolic periodic Fatou components growing outside as in \S\ref{level-1-construction}, with hierarchy $U^k_1(f),\dots,U^k_{n^k(f)}(f)$. Define
$$X^k_0(f)=\bigcup_{j=0}^{n^k(f)} \overline{U^k_j(f)}.$$ For each $n\geq 0$, define inductively $X_{n+1}^k(f)$ to be the connected component of $f^{-1}(X_n^k(f))$ containing $X_n^k(f)$. Then $$X_0^k(f)\subset X_1^k(f)\subset X_2^k(f)\subset\cdots.$$ 

As a natural generalization of Fatou components, the \emph{$k$-tree components} are defined to be the connected components of $f^{-n}(U^k_j(f))$. The \emph{limit point set} of a $k$-tree component $U$ is defined to be $\overline{U}\setminus U$. Each $X_n^k(f)$ is the closure of a finite union of $k$-tree components, of which any two touch at one limit point at most. 

Let $$Y^k(f)=\overline{\bigcup_{n\in\mathbb{N}} X_n^k(f)}$$ be the \emph{Fatou tree of level $k$}. 
Let $$Y_\infty^k(f)= Y^k(f)\setminus\bigcup_{n\in\mathbb{N}} X_n^k(f)$$ be the set of all limit points on $Y^k(f)$. 
See Notation \ref{notation-tree-k} for a list of notations related to $Y^k(f)$. 

\begin{example}
\label{example-crit}
Consider the one-parameter family of quartic polynomials $$f_c(z) = 2c c' z^2 - (4/3)(c+c')z^3 + z^4, \ c\in\mathbb{C},\ c(c^3-2)\neq 0,$$
where $c' = (c^6 - 2c^3 + 3)/(2c^2(c^3 - 2))$. 
Then $0,c,c'$ are critical points of $f_c$. We have $f_c(0)=0$, $f_c^2(c)=f_c(c)$ (i.e. $f_c(c)$ is a fixed point) and $c'$ is the free critical point. See Figure \ref{fig-critical}. 
The first two polynomials have Fatou trees of level 2, and the last one has Fatou trees of level at most one. 
\begin{figure}[htbp]
\centering
\includegraphics{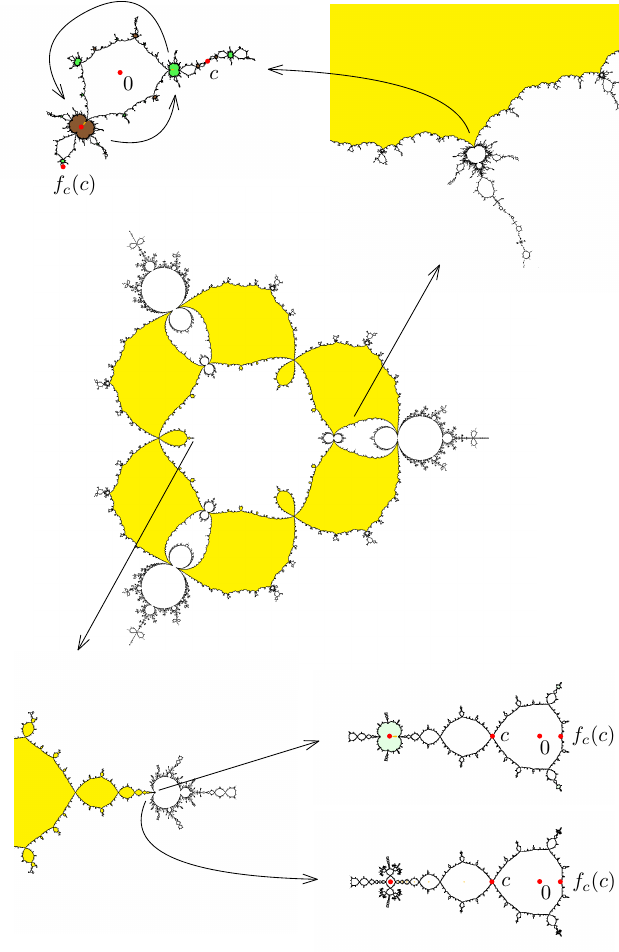}
\caption{The bifurcation locus of $\{f_c\mid c\in\mathbb{C}, c(c^3-2)\neq 0\}$ and three Julia sets with marked critical orbits. A polynomial $f_c$ with $c$ in the closure of the yellow part is a candidate for $\partial_{\rm reg}\mathcal H_4$.}
\label{fig-critical}
\end{figure}
\end{example}



%
%
%
%
%
%
%
%

\begin{proof}
[Proof of Theorem \ref{thm-tree-level-k}]
For $k=0$, the results follow from \cite{RY} and Lemma \ref{iff-diam-0}. 
For $k\geq1$, the proof use the same idea as the case $k=1$, and induction. 

See Proposition \ref{Y1-lc} for Property 1 and Corollary \ref{coro-limb} for Property 2. 
For the preperiodic case of Property 3, see Theorem \ref{dynam-prop}(\ref{dynam-prop-prep}) and Lemma \ref{U-prep}(\ref{U-prep-repel-parab}); for the wandering case of Property 3, see Corollary \ref{ray-num}. 
\end{proof}

Let $f\in \mathcal{Y}_d^k$ for some $k\geq0$. 
Note that $$2\leq \deg(f|_{Y^0(f)^\circ})<\deg(f|_{Y^1(f)^\circ})<\cdots<\deg(f|_{Y^k(f)^\circ})\leq d.$$
We have $k\leq d-2$. That is, $\mathcal{Y}_d^k=\emptyset$ when $k>d-2$. 

Let $f\in\mathcal{Y}_d^0$. Since 
$$\mathcal{Y}_d^0\supset\mathcal{Y}_d^1\supset\cdots\supset\mathcal{Y}_d^{d-2}\supset\mathcal{Y}_d^{d-1}=\emptyset,$$
there is a unique $0\leq k(f)\leq d-2$ such that $f\in \mathcal{Y}_d^{k(f)}\setminus \mathcal{Y}_d^{k(f)+1}$. 
Then $$Y^0(f) \subset Y^1(f)\subset\cdots \subset Y^{k(f)}(f) \subset K(f),$$
and $Y_\infty^{k(f)}(f)$ has neither parabolic point nor critical point. 
We call $Y^{k(f)}(f)$ the \emph{maximal Fatou tree} of $f$.

\begin{lemma}
\label{one-ray}
Let $f\in\mathcal{P}_d$ with an attracting fixed point at $0$. If $Y^{k(f)}(f)= K(f)$, then for any $y\in Y^{k(f)}_\infty(f)$, there is exactly one external ray landing at $y$. 
\end{lemma}

\begin{proof}
Note that $Y_\infty^{k(f)}(f)$ contains neither parabolic point nor critical point. 
By Theorem \ref{thm-tree-level-k}(\ref{thm-tree-level-k-ray}), we just need to consider the case that $y$ is a repelling periodic point. 

Since $Y^{k(f)}(f)= K(f)$, we have $L_y(K(f), Y^{k(f)}(f))=\{y\}$. 
By the same argument as in Lemma \ref{U-prep}, there is only one external ray landing at $y$. 
\end{proof}

\subsection{The maximal Fatou tree is a renormalization}
\label{subsec-tree-kf}

In this subsection, we will prove Theorem \ref{Ymax=Kf}: for any $f\in \partial_{\rm reg}\mathcal H_d$, we have $Y^{k(f)}(f)= K(f)$. 
We begin with a lemma that generalizes Corollary \ref{coro-poly-like}. 

%
%
%

\begin{lemma} 
\label{ren-maximal} 
For any $f\in\mathcal{Y}_d^0$, there is a polynomial-like renormalization $f: U \to V$ with filled Julia set $K(f|_U)=Y^{k(f)}(f)$. 
\end{lemma}

\begin{proof}
Write $Y=Y^{k(f)}(f)$ and $Y_\infty=Y^{k(f)}_\infty(f)$. 
Then $Y_\infty$ contains no critical point and no parabolic point. 
If $Y = K(f)$, the construction is trivial. In the following, we assume $Y \neq K(f)$. 
To make $Y$ into a small filled Julia set, we need to remove the nontrivial limbs. 

Let $y\in\partial Y$ with $L_y(K(f),Y)\neq\{y\}$. 
By Theorem \ref{thm-tree-level-k}(\ref{thm-tree-level-k-limb}), there are finitely many minimal sectors $S_1,\dots, S_n$ with root $y$ such that $L_y(K(f),Y)=K(f)\cap\bigcup_{1\leq j\leq n} \overline{S_j}$. Here we say a sector is minimal if it contains no external ray landing at its root. Let $\mathcal{A}_y$ denote the family of these minimal sectors. 

We claim that $d\theta\not\equiv d\theta'\pmod{\mathbb Z}$ for any $S_f(\theta,\theta')\in\mathcal{A}_y$. 
If $y\in Y_\infty$, this is true because $Y_\infty$ contains no critical point. 
If $y\in\partial Y\setminus Y_\infty$, assume $f(R_f(\theta)) = f(R_f(\theta'))$; let us find a contradiction. Since $y\in\partial Y\setminus Y_\infty$, we have $y\in X_n^{k(f)}(f)$ for some $n$ large enough. By $f(R_f(\theta)) = f(R_f(\theta'))$, both $X_{n+1}^{k(f)}(f)\cap S_f(\theta,\theta')$ and $X_{n+1}^{k(f)}(f)\cap S_f(\theta',\theta)$ are nonempty. Then $X_{n+1}^{k(f)}(f)\cap S_f(\theta,\theta')\neq\emptyset$ implies $Y\cap S_f(\theta,\theta')\neq\emptyset$. This is a contradiction. Furthermore, we have 
\begin{align*}
&\mathcal A_{f(y)} = \{S_f(d\theta,d\theta')\mid S_f(\theta,\theta')\in\mathcal A_y\},\\
&\#(\mathcal A_y)=\#(\mathcal A_{f(y)})\cdot\deg(f,y).
\end{align*}
Then $\crit(f)\cap \bigcup_{n\in\mathbb{N}} S_f(d^n\theta, d^n\theta') \neq \emptyset$; otherwise, the angular difference between $d^n\theta$ and $d^n\theta'$ grows exponentially as $n\rightarrow \infty$, which is impossible. 
Since there are at least two external rays landing at $f^n(y)$ for any $n\geq0$, the point $y$ is pre-repelling or pre-parabolic by Theorem \ref{thm-tree-level-k}(\ref{thm-tree-level-k-ray}).
 
Let 
\begin{align*}
\mathcal{A} &= \bigcup_{y\in\partial Y,\ L_y(K(f),Y)\neq\{y\}}\mathcal{A}_y, \\
\mathcal{C} &= \{S\in \mathcal{A}\mid S\cap \crit(f)\neq\emptyset\},\\
\mathcal{P} &= \{S_f(d^n\theta,d^n\theta')\mid S_f(\theta,\theta')\in\mathcal C, n\geq1\},\\
\mathcal{Q} &= \{S_f(\theta,\theta')\in\mathcal{A}\mid S_f(d\theta,d\theta')\in\mathcal P\}. 
\end{align*}
Then $\mathcal{Q}$ is finite and $\mathcal P\subset \mathcal Q$. Given $r>1$. Let $$W_\mathcal{Q} = (\mathbb{C}\setminus B_f^{-1}(\mathbb{C}\setminus r\mathbb{D}))\setminus \bigcup_{S\in\mathcal{Q}}\overline{S}.$$
Replacing $r$ and $\mathcal Q$ by $r^d$ and $\mathcal P$ respectively gives $W_\mathcal{P}$. 

By the thickening technique \cite{Mil-lc}, the proper map $f: W_\mathcal{Q}\to W_\mathcal{P}$ extends to a polynomial-like map $f:U\to V$ such that $U\cap\crit(f) = Y\cap \crit(f)$ and $Y\subset U$. Note that $Y_\infty$ contains no parabolic point. Every $S\in \mathcal{Q}$ with parabolic periodic root is in a repelling direction of this parabolic point, so the thickening procedure is still valid in this case. See Figure \ref{fig-polynomial-like}. 

\begin{figure}[ht]
\centering
\includegraphics{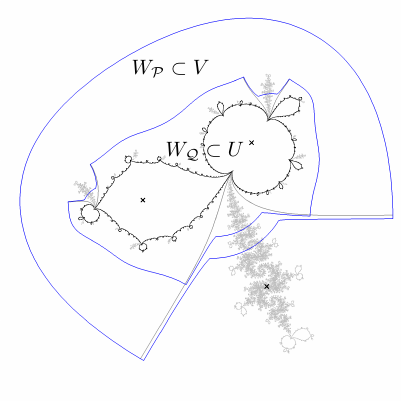}
\caption{The polynomial-like map $f:U\to V$}
\label{fig-polynomial-like}
\end{figure}

To finish we show $K(f|_U)=Y$. It follows from $f(Y) = Y$ that $Y\subset K(f|_U)$. 
If $S\in\mathcal{Q}$, by the construction of $f:U \to V$, every point in $S$ escapes under $f|_U$. 
If $S\in\mathcal{A}\setminus\mathcal{Q}$, there is a minimal $n\geq1$ such that $f^n(S)\in \mathcal{Q}$, so the points in $S$ also escape under $f|_U$. 
Therefore $K(f|_U)\subset K(f)\setminus\bigcup_{S\in\mathcal{A}} S = Y$, and then $K(f|_U)=Y$. 
\end{proof}

The above construction implies $\deg(f|_U)<d$ if $Y^{k(f)}(f)\neq K(f)$. Generally, we have the following fact. 

\begin{lemma}
\label{deg-less-d}
Let $f$ be a polynomial with degree at least two. If $f:U\to V$ is a polynomial-like restriction with $K(f|_U)\neq K(f)$, then $\deg(f|_U)<\deg(f)$. 
\end{lemma}

\begin{proof}
Assume $K(f|_U)\neq K(f)$ and $\deg(f|_U)=\deg(f)$; we will find a contradiction. 
Since $\deg(f|_U)=\deg(f)$, for any $z\in J(f|_U)$, we have $f^{-1}(z)\subset J(f|_U)$. Note that the inverse orbit of any $z\in J(f)$ is dense in $J(f)$. It follows that $J(f)\subset J(f|_U)$, and so $J(f)= J(f|_U)$. This contradicts $K(f|_U)\neq K(f)$. 
\end{proof}

\begin{proof}
[Proof of Theorem \ref{Ymax=Kf}]
Let $f\in \partial_{\rm reg}\mathcal H_d$ and suppose $Y^{k(f)}(f)\neq K(f)$; we will find a contradiction. By Lemma \ref{ren-maximal}, there is a polynomial-like map $f:W\to V$ with filled Julia set $K(f|_W)=Y^{k(f)}(f)$. 
By Lemma \ref{deg-less-d}, we have $\deg(f|_W)<d$. 

Furthermore, we require that $\partial W\cap \crit(f)=\emptyset$. Then there is a neighborhood $\mathcal N$ of $f$ so that for all $g\in \mathcal N$, we have $|g'(0)|<1$, the component $W_g$ of $g^{-1}(V)$ containing $0$ is compactly contained in $V$, and $g: W_g\to V$ is a polynomial-like map of degree $\deg(g|_{W_g})=\deg(f|_W)<d$. 
Let $U_g$ (resp. $U'_g$) denote the Fatou component of $g$ (resp. $g|_{W_g}$) containing $0$. 
Then $U_g = U'_g$. When $g\in \mathcal N\cap \mathcal H_d$, we have $$
\deg(g|_{W_g}) 
\geq \deg(g|_{U'_g}) 
= \deg(g|_{U_g}) 
= d.$$ This is a contradiction. 
Therefore $Y^{k(f)}(f) = K(f)$.  

Finally, the second assertion follows from Lemma \ref{one-ray}. 
\end{proof}

\begin{proof}
[Proof of Theorem \ref{J-lc}]
By Theorems \ref{Ymax=Kf} and \ref{thm-tree-level-k}(\ref{thm-tree-level-k-lc}). 
\end{proof}

\subsection{A counterexample}
\label{subsec-example}

Let $$\mathcal Y_d=\{f\in\mathcal Y_d^1\mid Y^{k(f)}(f)=K(f)\}.$$
By Theorem \ref{Ymax=Kf}, we have $ \partial_{\rm reg}\mathcal H_d\subset \mathcal Y_d$. But the converse may not hold. To see this, we will give examples of polynomials in $\mathcal Y_4\setminus \partial_{\rm reg}\mathcal H_4$. Roughly speaking, for a polynomial $f\in\mathcal{P}_d$, the condition $Y^{k(f)}(f)=K(f)$ is necessary but not sufficient for $f\in\partial_{\rm reg}\mathcal H_d$. 

\begin{example}
\label{example-resit}
Consider the one-parameter family of quartic polynomials $$f_a(z) = \left(a^2+\frac{2}{a}\right)z^2 - \left(2a+\frac{1}{a^2}\right)z^3 + z^4, \ a\in\mathbb{C}\setminus\{0\}.$$ Then $f_a(a)=a$ and $f_a'(a)=1$. 
If $a^3 = 1$, then $a$ is a triple fixed point of $f_a$. 
If $a^3\neq 1$, then $a$ is a double fixed point of $f_a$, and $$\res(f_a,a) = 1 - \frac{2}{a^3-1} - \frac{1}{(a^3-1)^2}.$$ For basic facts of the r\'esidu it\'eratif, see \cite[\S12]{Mil}. In Figure \ref{fig-parabolic}, the three shaded figure eights correspond to $\Re(\res(f_a,a))<0$ (i.e. $a$ is parabolic attracting), and their cross points correspond to the roots of $a^3=1$. 

For $a$ with $\Re(\res(f_a,a))<0$ and $|a|<1$, the parabolic point $a$ is on the boundary of $U_{f_a}(0)$ and the parabolic basin attached at $a$ contains two critical points, as showing in the lower right corner of Figure \ref{fig-parabolic}. Clearly, $Y^{k(f_a)}(f_a) = K(f_a)$ in this case. It follows that $f_a\in\mathcal Y_4$. Since $a$ is parabolic attracting, a hyperbolic polynomial in $\mathcal P_4$ close to $f_a$ has two simple fixed points near $a$, and at least one of them is attracting, thus $f_a\notin\partial \mathcal H_4$. 

As comparison, the polynomial in the lower left corner of Figure \ref{fig-parabolic} has a parabolic repelling fixed point, and it is on the boundary of $\mathcal H_4$. 
\begin{figure}[ht]
\centering
\includegraphics{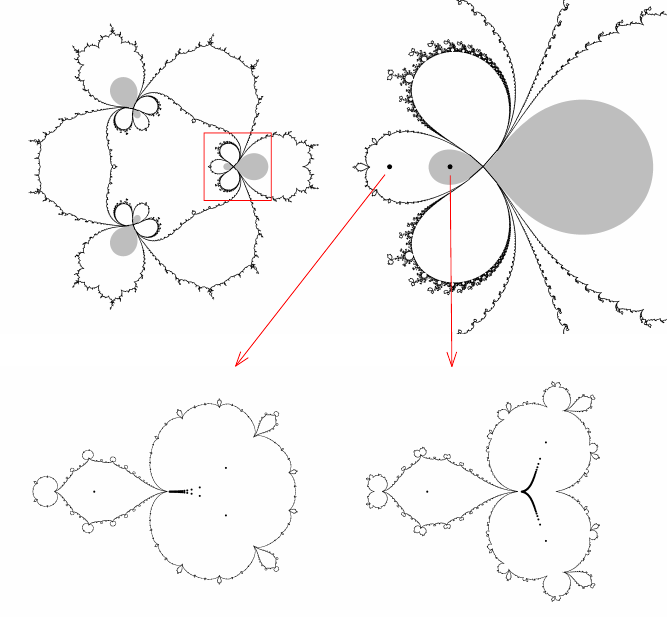}
\caption{The bifurcation locus of $\{f_a\mid a\in\mathbb{C}\setminus\{0\}\}$ and two Julia sets with marked critical orbits}
\label{fig-parabolic}
\end{figure}
\end{example}

\subsection{The parabolic case}
\label{subsec-parab}

Given $f\in\mathcal P_d$ with a parabolic fixed point at $0$ and connected Julia set. 
Let $U_f(0)$ be the union of all parabolic periodic Fatou components $U$ such that $\{f^n|_U\}_{n\in\mathbb{N}}$ tends to $0$. 
Then $U_f(0)$ consists of finitely many Jordan disks with a common boundary point $0$. 

If $U_f(0)$ has only one connected component, let $Y^0(f)=\overline{U_f(0)}$ be the Fatou tree of level $0$. 
Furthermore, if $\partial U_f(0)\setminus \{0\}$ contains critical points or parabolic periodic points, then we can define the Fatou tree $Y^1(f)$ of level $1$ as the Fatou tree of level $1$ for the attracting case. 

If $U_f(0)$ has at least two connected components, we can define the Fatou tree $Y^0(f)$ of level $0$ as the Fatou tree of level $1$ for the attracting case. 
(In this case, $\overline{U_f(0)}$ cannot be homeomorphic to the filled Julia set of any polynomial, but $Y^0(f)$ could potentially satisfy this property.)

Inductively, we will get the maximal Fatou tree $Y^{k(f)}(f)$. In the parabolic case, the maximal possibility of $k(f)$ is $d-1$, i.e. the number of critical points of $f$ in the plane. 

Using the same method, the analogues of Theorems \ref{J-lc} to \ref{thm-tree-level-k} still hold in the parabolic case. 

%
%
%

\section{Rigidity}
\label{sec-rigidity}

Let $f, g\in \mathcal P_d$ be two polynomials. We say $f$ and $g$ are \emph{hybrid conjugate}, if there is a homeomorphism $\phi: \mathbb C\to \mathbb C$, satisfying that
\begin{itemize}
\item $\phi\circ f(z)=g\circ \phi (z)$ for all $z\in \mathbb C$;

\item $\phi$ is holomorphic in the Fatou set. 
\end{itemize} 
Recall that 
$$\mathcal Y_d=\{f\in\mathcal Y_d^1\mid Y^{k(f)}(f)=K(f)\}.$$

\begin{theorem} 
\label{hybrid4Y}
Let $f, g\in \mathcal Y_d$. If $f$ and $g$ are hybrid conjugate, then they are affine conjugate. 
\end{theorem}

\begin{proof}
The proof will be divided into two parts: Propositions \ref{hybird2qc} and \ref{nilf}. 
\end{proof}

\begin{proof}
[Proof of Theorem \ref{hybrid4H}]
By Theorems \ref{Ymax=Kf} and \ref{hybrid4Y}, $f$ and $g$ are affine conjugate. By the normalization $\phi'(\infty)=1$, the affine conjugation has the form $z\mapsto z+b$. Since $0$ is the only attracting fixed point of $f$ and $g$ (in the plane), we have $b=0$. Therefore $f=g$. 
\end{proof}

\subsection{Quasiconformality}

\begin{proposition}
\label{hybird2qc}
Let $f, g\in \mathcal Y_d$. Suppose $f$ and $g$ are hybrid conjugate under $\phi$. Then $\phi$ is quasiconformal. 
\end{proposition}

\begin{proof}
The proof is based on the quasiconformality criterion stated in Proposition \ref{criterion-qc}. 

In \S\ref{sect-puzzle}, we construct a puzzle for $Y^1(f)$, where only the parabolic basins in $Y^1(f)$ are considered. 
Now for $Y^{k(f)}(f)=K(f)$, we let 
$$\Omega = \mathbb{C} \setminus  \bigcup\left\{\overline{\Omega_\infty},\overline{\Omega_0}, \overline{\Omega_U}\mid \text{parabolic periodic Fatou component $U\subset K(f)$}\right\},$$ 
and then define a graph $\Gamma$ in the same way. 
For $n\in\mathbb{N}$, a puzzle piece of depth $n$ is a connected component of $f^{-n}(\Omega\setminus \Gamma)$. 
A nest is a nested sequence $P_0\supset P_1\supset P_2\supset \cdots$ such that $P_n$ is a puzzle piece of depth $n$ for each $n\in\mathbb{N}$. 
Then the end $\bigcap_{n\in\mathbb{N}}\overline{P_n}$ of any nest is a singleton. 

Let $X_1$ be the Fatou set of $f$. 
Let $X_2$ (resp. $X_3$) consist of all wandering points in $J(f)$ with the bounded degree condition (resp. persistent recurrence condition). Let $X_4$ be the set of all preperiodic points of $f$. It is clear that $X_1$ and $X_4$ satisfy the conditions in Proposition \ref{criterion-qc}. Now let us verify the conditions about $X_2$ and $X_3$. 

\vspace{6pt}
\textbf{Claim 1.} $\operatorname{area}(X_2)=0$ and $\operatorname{area}(\phi(X_2))=0$. 
\vspace{6pt}

Let $x\in X_2$ and let $\mathbf{P}$ be the nest with end $\{x\}$. Then there exists a sequence $\{n_k\}_{k\in\mathbb{N}}$ tending to $\infty$ and a constant $D$ so that $\deg(f^{n_k}|_{P_{n_k}})\leq D$ for any $k\in\mathbb{N}$. 
By Lemma \ref{BD2elevator}, $\mathbf{P}$ also satisfies the elevator condition. 
After passing to a subsequence, we may assume $f^{n_k}(P_{n_k})$ is the same puzzle piece $Q_0$ and $f^{n_k}(x)$ tends to a point $y$ in $Q_0$. 
For each $k\in\mathbb{N}$, there are conformal maps $\alpha_k:Q_0\to\mathbb{D}$ with $\alpha_k(f^{n_k}(x)) = 0$ and $\beta_k:P_{n_k}\to\mathbb{D}$ with $\beta_k(x)=0$. Then $h_k:=\alpha_k\circ f^{n_k}\circ\beta_k^{-1}$ is a Blaschke product on the unit disk with $h_k(0)=0$ and $\deg(h_k)\leq D$. Furthermore, we can require $\alpha'_k(f^{n_k}(x)) > 0$. Then $\alpha_k$ converges uniformly to the conformal map $\alpha:Q_0\to\mathbb{D}$ with $\alpha(y)=0$ and $\alpha'(y)>0$. 

For $r\in(0,1)$, let $U_k(r)$ be the connected component of $(h_k\circ\beta_k)^{-1}(\mathbb{D}(r))$ containing $x$. 
By Lemma \ref{shape-distortion}, there is a constant $C=C(D,r)$ such that $$\operatorname{Shape}(U_k(r),x)\leq C.$$
Since $\alpha_k\to\alpha$ as $k\to\infty$, there exists a small disk $B(x_0,r_0)$ independent of $k$ and contained in $\mathbb{D}(r)\setminus \alpha_k(J(f))$. Then $$\frac{\operatorname{area}(\mathbb{D}(r)\cap \alpha_k(J(f)))}{\operatorname{area}(\mathbb{D}(r))}\leq 1-\frac{r_0^2}{r^2}=:1-\delta.$$ 
By \cite[Lemma 2.3]{Shen}, there is a constant $\varepsilon= \varepsilon(D,r,\delta)>0$ such that 
$$\frac{\operatorname{area}(U_k(r)\cap J(f))}{\operatorname{area}(U_k(r))}\leq 1-\varepsilon.$$
Letting $k\to\infty$, we see that $x$ is not a Lebesgue density point of $J(f)$. Hence $\operatorname{area}(X_2)=0$, and the same discussion gives $\operatorname{area}(\phi(X_2))=0$. 

\vspace{6pt}
\textbf{Claim 2.} $\underline{H}(\phi,x)<\infty$ for any $x\in X_2$. 
\vspace{6pt}

Let $\operatorname{dist}(z_1,z_2)=|z_1-z_2|$ be the Euclidean distance of $z_1$ and $z_2$ in $\mathbb{C}$. 
For a Jordan domain $U$, let $\operatorname{dist}_U(z_1,z_2)$ denote the hyperbolic distance of $z_1$ and $z_2$ in $U$. Furthermore, We define corresponding distances between planar sets by taking infimum. 

Given $k\in\mathbb{N}$. Let $D_k=\deg(f^{n_k}|_{P_{n_k}})$, and let $x_k^1=x, x_k^2, \dots, x_k^{D_k}$ denote the preimages of $f^{n_k}(x)$ under $f^{n_k}:P_{n_k}\to Q_0$. 
By \cite[Lemma 4]{Zhai}, there are constants $\varepsilon = \varepsilon(D)>0$ and $0<r_k<\frac{1}{2}{\operatorname{dist}(x,\partial P_{n_k})}$ such that $$\operatorname{dist}_{P_{n_k}}\left(S^1(x,r_k),\ \{x_k^1, x_k^2, \dots, x_k^{D_k}\}\right)\geq \varepsilon,$$
where $S^1(x,r_k)$ is the round circle centered at $x$ with radius $r_k$. 
For any $z\in S^1(x,r_k)$ and $1\leq j\leq D_k$, 
$$\operatorname{dist}_{\mathbb{D}}\left(\frac{\beta_k(z)-\beta_k(x_k^j)}{1-\overline{\beta_k(x_k^j)}\beta_k(z)}, 0\right) = \operatorname{dist}_{\mathbb{D}}(\beta_k(z),\beta_k(x_k^j))=\operatorname{dist}_{P_{n_k}}(z,x_k^j)\geq\varepsilon.$$
Note that $\operatorname{dist}_{\mathbb{D}}(r,0) = \log\frac{1+r}{1-r}$. For any $z\in S^1(x,r_k)$, we have 
$$|h_k\circ \beta_k(z)| = \prod_{j=1}^{D_k} \left|\frac{\beta_k(z)-\beta_k(x_k^j)}{1-\overline{\beta_k(x_k^j)}\beta_k(z)}\right|\geq\left( \frac{e^\varepsilon-1}{e^\varepsilon+1}\right)^{D_k}\geq\left( \frac{e^\varepsilon-1}{e^\varepsilon+1}\right)^D.$$
On the other hand, by $r_k<\frac{1}{2}{\operatorname{dist}(x,\partial P_{n_k})}$, we have $$\operatorname{mod}(P_{n_k}\setminus \overline{B(x,r_k)})\geq \operatorname{mod}(B(x,2r_k)\setminus \overline{B(x,r_k)})=\frac{\log 2}{2\pi}.$$
Hence the hyperbolic diameter of $S^1(x,r_k)$ in $P_{n_k}$  is uniformly bounded. By Schwarz's lemma, the hyperbolic diameter of $h_k\circ\beta_k(S^1(x,r_k))$ in $\mathbb{D}$ is also uniformly bounded. Now there are constants $0<R_1<R_2<1$ independent of $k$ such that $$h_k\circ \beta_k(S^1(x,r_k))\subset \{z\in\mathbb{D}\mid R_1<|z|<R_2\}.$$

Note that $\alpha_k\to\alpha$ as $k\to\infty$ and $\phi:Q_0\to\phi(Q_0)$ is a homeomorphism. 
There are constants $0<R'_1<R'_2$ independent of $k$ such that \begin{align*}
\phi\circ f^{n_k}(S^1(x,r_k))
&= \phi\circ\alpha_k^{-1} \circ h_k\circ\beta_k(S^1(x,r_k))\\
&\subset \left\{z\in\phi(Q_0)\mid R'_1<\operatorname{dist}_{\phi(Q_0)}(z,\phi\circ f^{n_k}(x))<R'_2\right\}.\end{align*}
Then by $\deg(g^{n_k}:\phi(P_{n_k})\to\phi(Q_0))\leq D$ and Lemma \ref{shape-distortion}, there is a constant $M$ independent of $k$ such that $\operatorname{Shape}(\phi(B(x,r_k)),\phi(x))\leq M$. Letting $k\to\infty$ gives $\underline{H}(\phi,x)<\infty$. 

\vspace{6pt}
\textbf{Claim 3.} There is a constant $M$ such that for each $x\in X_3$, the nest $\mathbf{P}$ of $x$ has a subnest $\{P_{n_k}\}_{k\in\mathbb{N}}$ satisfying 
$$\operatorname{Shape}(P_{n_k},x)\leq M \text{~and~} 
\operatorname{Shape}(\phi(P_{n_k}),\phi(x))\leq M$$
for any $k\in\mathbb{N}$. 
\vspace{6pt}

Fix $x\in X_3$ and let $\mathbf{P}$ be the nest with end $\{x\}$. 
Let $\mathbf{C}\in \omega(\mathbf{P})\cap \mathcal P_{\rm c}$. Then $\mathbf{C}$ has a subnest 
$$C_{n_0}
\supset
C_{n'_1}\supset C_{n_1}\supset C_{\widetilde n_1}
\supset
C_{n'_2}\supset C_{n_2}\supset C_{\widetilde n_2}
\supset
\cdots$$ 
satisfying the properties in Theorem \ref{p-nest}. 
Suppose $\operatorname{end}(\mathbf{C})=\{c\}$. 
Then $f^{n_{j+1}-n_j}(c)\in C_{n'_j}$ for each $j\geq 1$. It follows that 
$$f^{n_{j+1}-n_j}:(C_{n_{j+1}-n_j+n'_j},C_{n_{j+1}},C_{n_{j+1}-n_j+\widetilde n_j})\to(C_{n'_j}, C_{n_j}, C_{\widetilde n_j}).$$
Moreover, their degrees are the same number in $[2,D]$. 
By \cite[Lemma 2]{YZ}, \cite[Lemma 2.2]{RYZ} and Koebe's distortion theorem, there is a constant $K(m,D)\geq 1$ such that for any $j\geq j_0$, 
$$\operatorname{Shape}(C_{n_{j+1}},c)\leq K(m,D) \operatorname{Shape}(C_{n_j},c)^{\frac{1}{2}}.$$
Repeating this inequality, we have 
\begin{align*}
\operatorname{Shape}(C_{n_{j+1}},c)
&\leq K(m,D)^{1+\frac{1}{2}+\cdots + \frac{1}{2^{j-j_0}}} \operatorname{Shape}(C_{n_{j_0}},c)^{\frac{1}{2^{j+1-j_0}}}\\
&\leq K(m,D)^2\operatorname{Shape}(C_{n_{j_0}},c).
\end{align*}

Choose $j_1=j_1(x)\geq j_0$ such that for each $\mathbf{C}''\in \mathcal P_{\rm c}\setminus \omega(\mathbf{P})$, we have  $$f^k(P_{k+ n'_{j_1}})\cap  C''_{n'_{j_1}}=\emptyset$$ for any $k\geq0$ except at most one $k$ with $\sigma^k(\mathbf{P}) = \mathbf{C}''$. 
\footnote{
The relation $\sigma^k(\mathbf{P}) = \mathbf{C}''$ implies $f^k(x)$ is a critical point. We can also put such $x$ into the countable set $X_4$. 
}
For $j\geq j_1$, let $r_j\geq 1$ be the first entry time of $\mathbf{P}$ (or $x$) into $C_{\widetilde n_j}$. Then the degrees of the triple 
$$f^{r_j}:(P_{r_j+n'_j},P_{r_j+n_j},P_{r_j+\widetilde n_j})\to (C_{n'_j},C_{n_j},C_{\widetilde n_j})$$ 
are the same number in $[1,d^{d-1}]$ by Lemma \ref{first-entry-time}. 
By \cite[Lemma 2]{YZ}, \cite[Lemma 2.2]{RYZ} and Koebe's distortion theorem again, for any $j\geq j_1$, 
\begin{align*}
\operatorname{Shape}(P_{r_j+n_j},x)
&\leq K(m,d^{d-1}) \operatorname{Shape}(C_{n_j},c)\\
&\leq K(m,d^{d-1})K(m,D)^2\operatorname{Shape}(C_{n_{j_0}},c)\leq M,
\end{align*}
where $M$ is a constant independent of $x$. 

Since $\phi$ is a conjugacy, it preserves the combinatorial structure. Hence there is another constant $M'$ independent of $x$ such that $$\operatorname{Shape}(\phi(P_{r_j+n_j}),\phi(x))\leq M'$$ for $j$ large enough. 
This shows Claim 3, and completes the proof. 
\end{proof}

\begin{lemma}
\label{shape-distortion}
Let $U\Subset V$ be Jordan domains in $\mathbb{C}$, and let $x\in U$. Let $h:V\to \mathbb{D}$ be a proper map of degree $D$ with $h(x)=0$. Suppose there are constants $0<r\leq R<1$ such that $r\leq |h(z)|\leq R$ for any $z\in \partial U$. Then there is a constant $M=M(D,r,R)$ such that $\operatorname{Shape}(U,x)\leq M$. 

In particular, if $r=R$, then $\operatorname{Shape}(U,x)\leq M=M(D,r)$. 
\end{lemma}

\begin{proof}
Let $\alpha: V\to \mathbb{D}$ be a conformal map with $\alpha(x)=0$. 
By Schwarz's lemma, we have $\mathbb{D}(r)\subset \alpha(U)$.  
Let $W$ denote the connected component of $\alpha\circ h^{-1}(\mathbb{D}(R))$ containing $x$. 
Then 
$\operatorname{mod}(\mathbb{D}\setminus \overline{\alpha(U)})
\geq\operatorname{mod}(\mathbb{D}\setminus \overline{W})
\geq\operatorname{mod}(\mathbb{D}\setminus \overline{\mathbb{D}(R)})/D$.  
It follows that $\alpha(U)\subset \mathbb{D}(R')$ for some constant $R'=R'(D,R)<1$. 
Then $\operatorname{Shape}(\alpha(U),0)\leq \frac{R'}{r}$. 
Applying Koebe's distortion theorem to $\alpha^{-1}$ completes the proof. 
\end{proof}

\subsection{No invariant line field}

Let $f\in\mathcal{P}_d$. We say $f$ carries an \emph{invariant line field on its Julia set} if there is a measurable Beltrami differential $\mu=\mu(z) \frac{d\bar{z}}{dz}$ such that $f^*\mu = \mu$ a.e., $|\mu|=1$ on a positive area subset of $J(f)$ and $\mu = 0$ elsewhere.  

For a measurable function $\mu:\mathbb{C}\to\mathbb{C}$, we say $\mu$ is \emph{almost continuous} at $x$ if for any $\varepsilon>0$, $$\lim_{r\to0} \frac{\operatorname{area}(\{y\in B(x,r)\mid |\mu(y)-\mu(x)|\leq \varepsilon\})}{\operatorname{area}(B(x,r))}=1.$$
Then $\mu$ is almost continuous almost everywhere. 

\begin{proposition}
\label{nilf}
Any $f\in \mathcal Y_d$ carries no invariant line field on $J(f)$.  
\end{proposition}

\begin{proof}
Assume $f$ carries an invariant line field $\mu$ on $J(f)$; we will find a contradiction through \cite[Proposition 3.2]{Shen}. Continue the notations in the proof of Proposition \ref{hybird2qc}. 
Since $\mathbb{C}=\bigcup_{k=1}^4 X_k$, $\operatorname{area}(X_2\cup X_3)=0$ and $\mu=0$ on $X_1$, the line field $\mu$ is supported on $X_3$. Now let $x\in X_3$ such that its orbit $\{f^n(x)\mid n\in\mathbb{N}\}$ avoids the critical points of $f$.

As the proof of Claim 3 in Proposition \ref{hybird2qc}, 
for $j\geq j_1$, let $r_j\geq 1$ be the first entry time of $x$ into $C_{\widetilde n_j}$. 
Let $s_j$ be the smallest positive integer such that $f^{s_j}(P_{\widetilde n_j + r_j})$ contains a critical point $c'$, whose nest is denoted by $\mathbf{C}'$. 
See Figure \ref{fig-three-first-entries}. 
By the choices of $x$ and $j_1$, we have $\mathbf{C}'\in \omega(\mathbf{P})\cap \mathcal P_{\rm c}$. 
Note that $\mathbf{C}'\in \omega(\mathbf{C})$. 
Let $u_j\geq1$ be the first entry time of $c$ into $C'_{\widetilde n_j+r_j-s_j}$. 
Let $v_j\geq1$ be the first entry time of $x$ into $C_{\widetilde n_j+r_j-s_j+u_j}$. 
Define 
$$h_j = f^{-s_j}\circ f^{u_j+v_j}: P_{\widetilde n_j+r_j-s_j+u_j+v_j}\to P_{\widetilde n_j+r_j}.$$
By Lemma \ref{first-entry-time}, $2\leq\deg(h_j)\leq d^{2(d-1)}$. Furthermore, replacing the subscript $\widetilde n_j$ by $n_j$ or $n'_j$, we will get a proper map with the same degree.

Let $a\in f^{-v_j}(c)\cap P_{\widetilde n_j+r_j-s_j+u_j+v_j}$. Then $h'_j(a)=0$. 
We have shown $\operatorname{Shape}(P_{n_j+r_j},x)$ is uniformly bounded for $j\geq j_1$. The same reason gives the uniform bounds of $\operatorname{Shape}(P_{n_j+r_j-s_j+u_j+v_j},a)$ and $\operatorname{Shape}(P_{n_j+r_j},h_j(a))$. 

\begin{figure}[ht]
\centering
\includegraphics{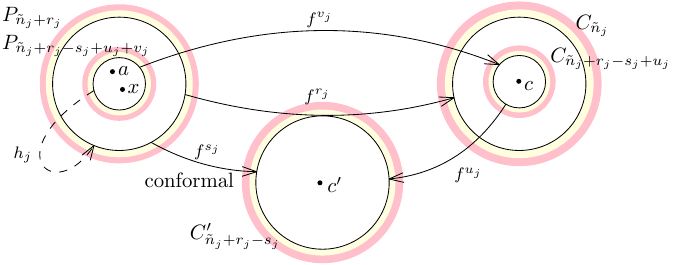}
\caption{A sketch of the construction of $h_j$. The adjacent three circles correspond to $n'_j<n_j<\widetilde n_j$.}
\label{fig-three-first-entries}
\end{figure}

By \cite[Proposition 3.2]{Shen}, $\mu(x)=0$ or $\mu$ is not almost continuous at $x$. 
This contradicts that $|\mu(x)|=1$ and $\mu$ is almost continuous at $x$ for $x$ in a positive area subset of $J(f)$. Therefore $f$ carries no invariant line field on its Julia set. 
\end{proof}

\appendix

\section{A criterion for quasiconformality}
\label{sec-qc-criterion}

In this appendix, we shall provide a criterion for quasiconformality which we used. 
This is an analogue of \cite[Lemma 12.1]{KSS}. 
See also \cite{Zhai}, \cite[Theorem 4]{Smania}, \cite{HK} and \cite{KK}. 

For a Jordan domain $U\subset\mathbb{C}$ and a point $x\in U$, the \emph{shape} of $U$ about $x$ is defined to be $$\operatorname{Shape}(U,x) = \frac{\max_{y\in\partial U}|y-x|}{\min_{y\in\partial U}|y-x|}.$$
We say $U$ has an $M$-shape if there is $x\in U$ so that $\operatorname{Shape}(U,x)\leq M$. 
For a homeomorphism $\phi:\Omega\to \Omega'$ between domains in $\mathbb{C}$ and $x\in \Omega$, let 
$$\underline{H}(\phi,x) = 
\liminf_{r\to 0}\ \operatorname{Shape}(\phi(B(x,r)),\phi(x))\in[1,\infty].$$
Here $B(x,r)$ is the open disk centered at $x$ with radius $r$. 

\begin{remark}
$\underline{H}(g\circ\phi,x) = \underline{H}(\phi,x)$ for a conformal transformation $g$ near $\phi(x)$. 
But $\underline{H}(\phi\circ g,x) = \underline{H}(\phi,x)$ may not hold. 
For example, consider a radial stretching defined by $\phi(z) = e^{-\frac{1}{|z|^2}}\cdot \frac{z}{|z|}$ and $\phi(0)=0$. Let $g(z)=\frac{z}{1-z}$. Then $\underline{H}(\phi,0)=1$, but $\underline{H}(\phi\circ g,0)=\lim_{r\to 0}\frac{|\phi\circ g(r)|}{|\phi\circ g(-r)|} = \infty$. 
\end{remark}

\begin{proposition}
\label{criterion-qc}
Let $\phi:\Omega\to \Omega'$ be an orientation-preserving homeomorphism between domains in $\mathbb{C}$. 
Let $H\geq1$ and $M\geq1$ be constants. 
Suppose $\Omega=\bigcup_{k=1}^4 X_k$ satisfy the following properties. 
\begin{enumerate}
\item $\underline{H}(\phi,x)\leq H$ for any $x\in X_1$.

\item $\operatorname{area}(\phi(X_2))=0$ and $\underline{H}(\phi,x)<\infty$ for any $x\in X_2$. 


\item There is a Markov family $\mathcal{F}$ of Jordan domains in $\mathbb{C}$ such that 
\footnote{
Here the Markov property of $\mathcal F$ means that $A_1\cap A_2=\emptyset$, or $A_1\subset A_2$, or $A_2\subset A_1$ for any $A_1,A_2\in \mathcal F$, and there are at most finitely many $A\in\mathcal{F}$ containing $A_0$ for any $A_0\in\mathcal{F}$. 
}
\begin{itemize}
\item for any $x\in X_3$, there is a sequence $A_1\Supset A_2\Supset \cdots$ in $\mathcal{F}$ such that $\bigcap_{n=1}^\infty A_n = \{x\}$; and
\item $\mathcal{F}$ and $\{\phi(A)\mid A\in\mathcal{F}\}$ have $M$-uniform shape. 
\end{itemize} 

\item $\operatorname{length}(\phi(X_4))=0$.
\end{enumerate} 
Then $\phi$ is $K$-qc, where $K=K(H,M)\geq1$. 
\end{proposition}

In complex dynamics, we can use this criterion to prove that topological conjugacy implies quasiconformal conjugacy. 
In applications, $X_1$ is the Fatou set, where $\phi$ is qc; $X_2$ consists of points in the Julia set with the bounded degree property; $X_3$ consists of points in the Julia set related to the persistently recurrent critical points, and we can construct $\mathcal{F}$ by KSS nests; and $X_4$ consists of preperiodic points, which is countable. 

\begin{proof}
Let $Q = Q(a,b,c,d)$ be a topological quadrilateral compactly contained in $\Omega$, and $Q'=Q'(a',b',c',d')=\phi(Q)$. 
Let $\Gamma$ be the family of rectifiable curves lying inside $Q$ which join the arcs $ab$ and $cd$, and let $\lambda$ denote the extremal length of $\Gamma$. Similarly, we can define $\Gamma'$ and $\lambda'$ for $Q'$. In order to prove the quasiconformality, we have to show there is a constant $K=K(H,M)\geq1$ such that $\lambda\geq \lambda'/K$. By \cite[Lemma 1]{Kelingos}, we only need to consider the case $\lambda'=1$. By the Riemann mapping theorem, there is a unique homeomorphism $g:\overline{Q'}\to\overline{\mathbb{D}}$ such that $g|_{Q'}$ is conformal and $g(a')=1,g(b')=i,g(c')=-1$. 
Then $g(d')=-i$ by $\lambda'=1$. 

We will construct a cover $\mathcal{A}_k$ of $X_k\cap Q$ for $1\leq k\leq 3$, and then define a metric $\rho$. A uniform lower bound of $\int_{\gamma}\rho(z)|dz|$ for any $\gamma\in\Gamma$ and an upper bound of $\int_\mathbb{C}\rho^2 dxdy$ give a lower bound of $\lambda$, which completes the proof. 

(1) Let us first define $\mathcal{A}_1$. For each $x\in X_1\cap Q$, there is an $r_x>0$ such that $g\circ\phi(B(x,r_x))$ has $2H$-shape and ${\rm mod}(Q\setminus \overline{B(x,r_x)})\geq 1$. 
By the Besicovitch covering theorem, there is a countable subfamily $\mathcal{A}_1$ of $\{B(x,r_x)\}_{x\in X_1\cap Q}$ such that $$\chi _{X_1\cap Q}\leq \sum_{A\in\mathcal{A}_1} \chi_A\leq C,$$ where $C$ is a universal constant. 
Then $$\sum_{A\in\mathcal{A}_1} \operatorname{diam}(g\circ \phi(A))^2\leq\sum_{A\in\mathcal{A}_1}\frac{(4H)^2}{\pi}\operatorname{area}(g\circ\phi(A))\leq 16 H^2 C. $$ 

(2) To define the cover $\mathcal{A}_2$, we first decompose the set $X_2\cap Q$ into countably many disjoint subsets $X_2^n = \{x\in  X_2\cap Q\mid n\leq \underline{H}(\phi,x)<n+1\}$ for integers $n\geq1$. Then fix a small open neighborhood $U_n$ of $X_2^n$ such that 
$$\operatorname{area}(g\circ\phi(U_n))\leq (n+1)^{-4}.$$
For each $x\in X_2^n$, there is an $r_x>0$ such that $g\circ\phi(B(x,r_x))$ has $(n+1)$-shape, $B(x,r_x)\subset U_n$ and ${\rm mod}(Q\setminus \overline{B(x,r_x)})\geq 1$. 
Let $\mathcal{D}_n = \{B(x,r_x)\}_{x\in X_2^n}$. 
Applying the Besicovitch covering theorem once again, there is a countable subfamily $\mathcal{A}_2$ of $\bigcup_{n=1}^\infty \mathcal{D}_n$ such that $$\chi _{X_2\cap Q}\leq \sum_{A\in\mathcal{A}_2} \chi_A\leq C,$$ where $C$ is the universal constant as before. 
Then 
\begin{align*}
&\sum_{A\in\mathcal{A}_2} \operatorname{diam}(g\circ \phi(A))^2\\
& = \sum_{n=1}^\infty \sum_{A\in\mathcal{A}_2\cap \mathcal{D}_n} \operatorname{diam}(g\circ \phi(A))^2
\leq \sum_{n=1}^\infty \sum_{A\in\mathcal{A}_2\cap \mathcal{D}_n} \frac{4(n+1)^2}{\pi}\operatorname{area}(g\circ \phi(A))
\\
&\leq \sum_{n=1}^\infty  \frac{4(n+1)^2}{\pi} C\cdot \operatorname{area}(g\circ\phi(U_n))
= \frac{4C}{\pi} \sum_{n=1}^\infty \frac{1}{(n+1)^2}
\leq \frac{4C}{\pi}. 
\end{align*}


(3) By the Markov property of $\mathcal{F}$, there is $\mathcal{A}_3\subset\mathcal{F}$ such that 
\begin{itemize}
\item $\mathcal{A}_3$ covers $X_3\cap Q$ consisting of pairwise disjoint Jordan domains; and
\item ${\rm mod}(Q\setminus \overline{A})\geq 1$ and ${\rm mod}(Q'\setminus \phi(\overline{A}))\geq 1$ for any $A\in\mathcal{A}_3$. 
\end{itemize}
Note that $\phi(A)$ has $M$-shape and ${\rm mod}(Q'\setminus \phi(\overline{A}))\geq 1$ for any $A\in\mathcal{A}_3$.   
By the Koebe distortion theorem, the family $\{g\circ\phi(A)\mid A\in\mathcal{A}_3\}$ has $N$-uniform shape, where $N$ is a constant depending only on $M$. It follows that
$$\sum_{A\in\mathcal{A}_3} \operatorname{diam}(g\circ \phi(A))^2\leq\sum_{A\in\mathcal{A}_3}\frac{4N^2}{\pi}\operatorname{area}(g\circ\phi(A))\leq 4N^2.$$

Now we have an open covering $\mathcal{A}:=\bigcup_{k=1}^3\mathcal{A}_k$ of $\bigcup_{k=1}^3 X_k$ satisfying that 
\begin{itemize}
\item $\sum_{A\in\mathcal{A}} \chi_A\leq C_1$;
\item $\sum_{A\in\mathcal{A}} \operatorname{diam}(g\circ \phi(A))^2 \leq C_2$; and 
\item $A$ has $M$-shape and ${\rm mod}(Q\setminus \overline{A})\geq 1$ for any $A\in \mathcal{A}$, 
\end{itemize}
where $C_1=C+C+1$ and $C_2 = C_2(H,M)
=16 H^2 C+\frac{4C}{\pi}+4N^2$. 

For each $A\in\mathcal{A}$, choose a point $x_A\in A$ so that $A':=B(x_A,\frac{\operatorname{diam}(A)}{2M})\subset A$, 
and let $\hat A = B(x_A,2 \operatorname{diam}(A))$. 
By ${\rm mod}(Q\setminus \overline{A})\geq 1$ and \cite[Lemma 1]{YZ}, for each rectifiable curve $\gamma$ that joins two points in $\partial Q$ and intersects some $A\in\mathcal{A}$, there is a universal constant $C_3>0$ such that $\operatorname{length}(\gamma \cap \hat A)\geq C_3 \operatorname{diam}(A)$. Define 
$$\rho = \sum_{A\in\mathcal{A}} a_A\chi_{\hat A},\ \ a_A = \frac{\operatorname{diam}(g\circ\phi(A))}{\operatorname{diam}(A)}.$$ 
It follows from $\operatorname{length}(\phi(X_4))=0$ that $\operatorname{length}(g\circ\phi(X_4))=0$. 
For any $\gamma\in\Gamma$, 
\begin{align*}
&\int_\gamma \rho(z)|dz| 
= \sum_{A\in\mathcal{A}} a_A\int_\gamma \chi_{\hat A}|dz| 
\geq \sum_{A\in\mathcal{A},A\cap\gamma\neq\emptyset} a_A\int_\gamma \chi_{\hat A}|dz| \\
&= \sum_{A\in\mathcal{A},\ A\cap\gamma\neq\emptyset} \operatorname{diam}(g\circ\phi(A)) \frac{\operatorname{length}(\gamma\cap \hat A)}{\operatorname{diam}(A)}
\geq C_3 \sum_{A\in\mathcal{A},\ A\cap\gamma\neq\emptyset} \operatorname{diam}(g\circ\phi(A))\\
&\geq C_3 \operatorname{length}(g\circ\phi(\gamma)\setminus g\circ\phi(X_4))
=C_3 \operatorname{length}(g\circ\phi(\gamma))
\geq \sqrt{2} C_3. 
\end{align*}
By \cite[Lemma 4.2]{Bojarski}, there is a constant $C_4$ depending only on $\frac{\operatorname{diam}(\hat A)}{\operatorname{diam}(A')} = 4M$ such that 
\begin{align*}
&\int_\mathbb{C} \rho^2 dxdy 
= \int\bigg(\sum_{A\in\mathcal{A}} a_A\chi_{\hat A}\bigg)^2 
\leq C_4 \int\bigg(\sum_{A\in\mathcal{A}} a_A\chi_{A'}\bigg)^2 \\
& \leq C_4 \int\bigg(\sum_{A\in\mathcal{A}} a_A\chi_{A}\bigg)^2
\leq C_4 \int \sum_{A\in\mathcal{A}} a_A^2\chi_{A}\sum_{A\in\mathcal{A}} \chi_{A}
\leq C_1 C_4\int \sum_{A\in\mathcal{A}} a_A^2\chi_{A}\\
&= C_1 C_4\sum_{A\in\mathcal{A}} a_A^2 \operatorname{area}(A) = C_1 C_4\sum_{A\in\mathcal{A}} \operatorname{diam}(g\circ \phi(A))^2 \frac{\operatorname{area}(A)}{\operatorname{diam}(A)^2}\\
&\leq \frac{\pi}{4} C_1 C_4 \sum_{A\in\mathcal{A}} \operatorname{diam}(g\circ \phi(A))^2
\leq \frac{\pi}{4} C_1 C_2 C_4.  
\end{align*}
Therefore $\lambda\geq \frac{8C_3^2}{\pi C_1 C_2 C_4}$. Letting $K = \frac{\pi C_1 C_2 C_4}{8C_3^2}$ completes the proof. 
\end{proof}

\bibliographystyle{alpha} 
\bibliography{references}

\end{document}